\documentclass[11pt]{article}
\usepackage{amsthm, amsmath, amssymb, amsfonts, url, booktabs, setspace, fancyhdr, bm}
\usepackage{pgf,tikz,tkz-graph}
\usetikzlibrary{math}
\usepackage[hidelinks]{hyperref}
\usepackage{geometry}
\usepackage{hyperref, enumerate}
\usepackage[shortlabels]{enumitem}
\usepackage[babel]{microtype}
\usepackage[english]{babel}
\usepackage[capitalise]{cleveref}
\usepackage{bbm}
\usepackage{csquotes}
\usepackage{mathrsfs}
\usepackage{mathabx}

\usepackage[normalem]{ulem}

\usepackage[T1]{fontenc}
\usepackage{amsfonts}
\usepackage{amscd}
\usepackage{graphicx}
\usepackage{enumitem}
\usepackage{verbatim}
\usepackage{hyperref}
\usepackage{amsmath,caption}
\usepackage{pdfpages,xcolor,framed,color}
\usepackage{todonotes}

\PassOptionsToPackage{hyphens}{url}

\newtheorem{thr}{Theorem}[section]

\newtheorem{lem}[thr]{Lemma}
\newtheorem{prop}[thr]{Proposition}
\newtheorem{conj}[thr]{Conjecture}
\newtheorem{cor}[thr]{Corollary}
\theoremstyle{definition}

\newtheorem{rem}[thr]{Remark}

\newtheorem{claim}[thr]{Claim}

\newcommand*{\myproofname}{Proof}
\newenvironment{claimproof}[1][\myproofname]{\begin{proof}[#1]}{\end{proof}}

\DeclareMathOperator{\mad}{mad}

\newcommand*{\sH}{\mathscr{H}}
\let\le\leqslant
\let\ge\geqslant
\let\leq\leqslant

\let\geq\geqslant

\usepackage{bookmark}

\newcommand{\abs}[1]{\left|#1\right|}

\renewcommand*{\phi}{\varphi}
\renewcommand*{\emptyset}{\varnothing}

\renewcommand{\Pr}{\mathop\mathbb{P}\nolimits}

\renewcommand{\a}{\alpha}

\newcommand\bb[1]{\textcolor{blue}{\textbf{#1}}}

\def\G{\mathcal{G}}

\title{Disjoint list-colorings for planar graphs}

\author{
	Stijn Cambie
 \thanks{Department of Computer Science, KU Leuven Campus Kulak-Kortrijk, 8500 Kortrijk, Belgium. 
 Supported by a FWO grant with grant number 1225224N. Email: \protect\href{mailto:stijn.cambie@hotmail.com}{\protect\nolinkurl{stijn.cambie@hotmail.com}}.}	
	\and
	Wouter Cames van Batenburg%
	\thanks{Delft Institute of Applied Mathematics, Delft University of Technology, Netherlands.
		Email: \protect\href{mailto:w.p.s.camesvanbatenburg@tudelft.nl}{\protect\nolinkurl{w.p.s.camesvanbatenburg@tudelft.nl}}.}
	\and
	Xuding Zhu%
	\thanks{School of Mathematical Sciences, Zhejiang Normal University, China. Supported by grants: NSFC 12371359, U20A2068. Email: \protect\href{mailto:xdzhu@zjnu.edu.cn}
{\protect\nolinkurl{xdzhu@zjnu.edu.cn}}.}
}

\date{\today}

\begin{document}

\maketitle

\begin{abstract}
 One of Thomassen's classical results is that every planar graph of girth at least $5$ is 3-choosable.
 One can wonder if for a planar graph $G$ of girth sufficiently large and a $3$-list-assignment $L$, one can do even better. 
 Can one find $3$ disjoint $L$-colorings (a packing), or $2$ disjoint $L$-colorings, or a collection of $L$-colorings that to every vertex assigns every color on average in one third of the cases (a fractional packing)?
 We prove that the packing is impossible, but two disjoint $L$-colorings are guaranteed if the girth is at least $8$, and a fractional packing exists when the girth is at least $6.$
 
 For a graph $G$, the least $k$ such that there are always $k$ disjoint proper list-colorings whenever we have lists all of size $k$ associated to the vertices is called the list packing number of $G$.
 We lower the two-times-degeneracy upper bound for the list packing number of planar graphs of girth $3,4$ or $5$.
 As immediate corollaries, we improve bounds for $\epsilon$-flexibility of classes of planar graphs with a given girth.
 For instance, where previously Dvo\v{r}\'{a}k et al. proved that planar graphs of girth $6$ are (weighted) $\epsilon$-flexibly $3$-choosable for an extremely small value of $\epsilon$, we obtain the optimal value $\epsilon=\frac{1}{3}$. 
Finally, we completely determine and show interesting behavior on the packing numbers for $H$-minor-free graphs for some small graphs $H.$
%Finally, we determine and show interesting behavior on the packing numbers for $H$-minor-free graphs, for $H \in\{K_{4},K_{5},K_{2,3},K_{2,4}\}$.
\end{abstract}

\section{Introduction}\label{sec:intro}
Given a list-assignment $L:V(G) \rightarrow 2^\mathbb{N}$ to a graph $G$, a (proper) \emph{$L$-coloring} is a proper coloring of $G$ such that every vertex $v$ receives a color from its list $L(v)$. The \emph{list-chromatic number} (or \emph{choosability}) $\chi_{\ell}(G)$ of a graph is the minimum integer $k$ such that every list-assignment $L$ with lists of size at least $k$ admits an $L$-coloring. One of the most salient appearances of this parameter is in Thomassen's famous proof of the $5$-choosability of planar graphs~\cite{Thomassen94}.
Indeed, while every planar graph is $4$-colorable~\cite{AH76, RSST96}, and every planar triangle-free graph is $3$-colorable~\cite{G59}, this is no longer true in the list-coloring setting. The following theorem summarizes the sharp bounds for the choosability of planar graph under various girth constraints.

\begin{thr}\label{thm:list_planar}
 There exists a proper $L$-coloring for a planar graph $G$ if 
 \begin{itemize}
 \item $L$ is a $5$-list-assignment (\cite{Thomassen94}, \cite{Voigt93, Mirzakhani96}),
 \item $L$ is a $4$-list-assignment and $G$ is triangle-free (degeneracy+1, \cite{Voigt95}),
 \item $L$ is a $3$-list-assignment and $G$ has girth at least $5$ (\cite{Thomassen95}, odd cycle). 
 \end{itemize}
\end{thr}

A common generalisation of these three statements was recently obtained by Postle and Smith-Roberge~\cite{PS22}. Other sufficient conditions for $3$-choosability of planar graphs were proven in e.g.~\cite{DvPo18, GW11, DLS10} and citing papers therein. 
As it turns out, beyond the guaranteed existence of some proper $L$-coloring, one can actually prove that there are exponentially many for the mentioned graph classes.

\begin{thr}\label{thm:exponential}
 A planar graph $G$ has exponentially many $L$-colorings if 
 \begin{itemize}\itemsep0em 
 \item $L$ is a $5$-list-assignment (\cite{Thomassen07}),
 \item $L$ is a $4$-list-assignment and $G$ is triangle-free (\cite{KL18}),
 \item $L$ is a $3$-list-assignment and $G$ has girth at least $5$ (\cite{Thomassen2007}). 
 \end{itemize}
 That is, there is a constant $c>1$ such that every graph $G$ of order $n$ admits at least $c^{n}$ proper colorings in the above situations.
\end{thr}

It is well-known that for every graph $G$, every list-assignment $L$ with all lists of size $\geq \chi_{\ell}(G)+1$ admits exponentially many $L$-colorings. This can be seen by double-counting the pairs of a sublist-assignment $L_0$ of size $\chi_{\ell}(G)$ and an $L_0$-coloring. Thus the strength of 
Theorem~\ref{thm:exponential} lies in having the list sizes one smaller than trivial.
Our goal in this paper is to admit slightly larger list sizes (if necessary) in return for a much more structured (not always stronger\footnote{Let $G$ be the square of a path, and take $L(v)=\{1,2,3\}$ for all $v \in V(G)$. Then there are only $6$ proper $L$-colorings, while $\chi_{\ell}^\star(G)=\chi_{\ell}(G)=3.$})
property: having multiple disjoint $L$-colorings.
For example, any $3$-list-assignment $L$ of a planar graph $G$ with girth at least $8$ admits two disjoint proper $L$-colorings, as we will see in~\cref{thr:2disjcol_g8}.

Having the maximum number of possible disjoint $L$-colorings is captured by the notion of \emph{list-packing}.
In particular, we study the \emph{list packing number} $\chi_{\ell}^{\star}(G)$ of planar graphs $G$ under various girth constraints. The list packing number is a graph invariant that has recently been introduced in~\cite{CCDK21} as a fertile middle ground between the classic notions of list-coloring and strong coloring, and has since been further explored in~\cite{CCDK23,CH23,M23,KMMP22}. It is defined as the smallest integer $k$ such that for every list-assignment $L$ with all lists of size $k$, there exists a \emph{partition} into $k$ disjoint $L$-colorings, meaning that at every vertex $v$, every element of $L(v)$ appears in exactly one of these colorings. If such a partition exists it is called an \emph{$L$-packing}. 
We refer the reader to~\cite{CCDK21, CCDK23} for an extensive discussion of ties with other coloring parameters. Since an $L$-packing requires at least one $L$-coloring we always have $\chi_{\ell}(G) \leq \chi_{\ell}^{\star}(G)$. On the other hand, greedy upper bounds that are immediate for list-coloring are nontrivial for list-packing, e.g. it is unknown whether $\chi_{\ell}^{\star}(G) \leq (1+o(1))\Delta$, as the maximum degree $\Delta$ of $G$ goes to infinity~\cite{CCDK23}. The main open problem is whether $\chi_{\ell}^{\star}(G)\leq C \cdot \chi_{\ell}(G)$ for some fixed $C>1$.
%There are arbitrarily large graphs $G$ with some list-assignments $L$ admitting exponentially many $L$-colourings but no $L$-packing~\cite{CCDK23}, and conversely there exist arbitrarily large graphs and $L$ admitting an $L$-packing but only a uniformly bounded number of $L$-colourings. To see the latter: for $G$ the square of a path and $L(v)=\{1,2,3\}$ for all $v \in V(G)$, there are only $6$ proper $L$-colorings while $\chi_{\ell}^\star(G)=\chi_{\ell}(G)=3.$
As it turns out, Theorem~\ref{thm:list_planar} remains true for \emph{correspondence-coloring}, which is a modern generalization of list-coloring introduced in~\cite{DvPo18}. Recently, Postle and Smith-Roberge~\cite{PS23} demonstrated that the first and third result in Theorem~\ref{thm:exponential} also survive in the correspondence setting, the first confirming a conjecture of Langhede and Thomassen~\cite{LT21}. While the definition of correspondence-colorings requires some effort to understand at first, ultimately it can be easier to work with as it naturally points the way to efficient induction proofs. This has motivated to similarly define~\cite{CCDK21} the \emph{correspondence packing number} $\chi_{c}^{\star}(G)$ of a graph $G$. For the precise definition, we refer the reader to Section~\ref{sec:notation}. For now it suffices to know that always
 $\chi_{\ell}^{\star}(G)\leq \chi_{c}^{\star}(G).$
 %In our proofs, we show that in fact Theorem~\ref{thm:listpacking} holds as well when replacing $\chi_{\ell}^{\star}(G)$ with $\chi_{c}^{\star}(G)$.\\
Regarding planar graphs, we prove:
\begin{thr}\label{thm:listpacking}
 For a planar graph $G$, the correspondence packing number is bounded as follows. 
 $$ \chi_{c}^{\star}(G) \leq \begin{cases}
 8 & \\
 5 & \text{ if } G \text{ is triangle-free} \\
 4 & \text{ if } G \text{ has girth at least } 5.\\
\end{cases}
$$
Moreover, for every $g \geq 5$, there exists a planar graph $G$ with girth $g$ and list packing number $\chi_{\ell}^{\star}(G)=4$.
\end{thr}

We stress that the optimal value $4$ for planar graphs with girth at least five deviates from the value $3$ found in Theorems~\ref{thm:list_planar} and~\ref{thm:exponential}. See section~\ref{sec:example_atleast4} for our construction. For planar graphs without girth constraints, we are not aware of any better lower bound than $5$. 
For triangle-free planar graphs, the best possible bound could be $4$ or $5$.

%comparison with trivial upper bounds following from chi_c^{\star}<= 2*degeneracy:
% planar =5-degenerate --> <=10
% triangle-free planar =3-degenerate --> <=6
% girth 5 planar =...3-degenerate? --> still<=6
% girth 6 planar = 2-degenerate --> <=4

Before discussing our strengthenings in further detail, we would like to draw attention to another line of works that we improve upon: 
in another recent direction, Dvo\v{r}\'{a}k, Norin and Postle~\cite{DNP19} introduced the notion of \emph{flexibility}. Rather than being satisfied with the existence of some $L$-colorings, one would like to choose preferred colors at each vertex and then be able to obtain an $L$-coloring that satisfies a large share of those preferences. 
More precisely, $G$ is \emph{$\epsilon$-flexible} with respect to $L$ if for every $S\subseteq V(G)$ and every collection of requests $(R(v)\in L(v))_{v\in S}$ on $S$, there exists an $L$-coloring $c$ of $G$ such that $c(v)=R(v)$ for at least $\epsilon |S|$ vertices $v\in S$.
If all lists are equal and of size $k$, 
then this is easily satisfied, as one can take an arbitrary $L$-coloring and uniformly permute the colors to obtain another coloring that satisfies at least a fraction $\epsilon=1/k$ of the requests. For general $L$ however, this approach does not work. This motivated the definition of the \emph{flexible chromatic number }$\chi_{\text{flex}}(\mathcal{G})$ of a graph class $\mathcal{G}$, which is the smallest integer $k$ such that there is an $\epsilon > 0$ for which every graph $G \in \mathcal{G}$ is $\epsilon$-flexible with respect to all $k$-fold list-assignments $L$. In the literature, often the best upper bounds on $\chi_{\text{flex}}(\mathcal{G})$ are obtained via analyzing a weighted variant $\chi_{\text{wflex}}(\mathcal{G})$ called the \emph{weighted flexibility}. As remarked in~\cite{DNP19} via a duality argument, this weighted parameter can equivalently be defined in terms of probability distributions: it is the smallest integer $k$ for which there exists an $\epsilon \in (0,1]$ such that for every $G\in \mathcal{G}$ and every $L$ with lists of size $\geq k$, there exists a probability distribution on $L$-colorings of $G$ such that every vertex $v$ is colored $c$ with probability $\geq \epsilon$, for every $c\in L(v)$. If this holds for a given $\epsilon$, we say that $G$ is \emph{weighted $\epsilon-$flexible}. 
Because for any collection of requests on any subset $S\subseteq V(G)$, the expected value of the number of satisfied requests is at least $\epsilon |S|$, it follows that $\chi_{\text{flex}}(\mathcal{G}) \leq \chi_{\text{wflex}}(\mathcal{G})$ always. 
%More precisely, this follows by computing the expected value of the number of satisfied requests

We now present a summary of what was the state-of-the art regarding flexibility of some important classes of planar graphs. For a more extensive survey we refer to~\cite{CCFHMM22}.

\begin{thr}\label{thm:flexibility}
 A planar graph $G$ is weighted $\epsilon$-flexible with respect to a list-assignment $L$ if 
 \begin{itemize}
 \item $L$ is a $7$-list-assignment and $\epsilon=7^{-36}$ (\cite[Thm.~2]{DNP19}),
 \item $L$ is a $4$-list-assignment and $\epsilon=2^{-186}$, and $G$ is triangle-free (\cite{DMMP21}),
 \item $L$ is a $3$-list-assignment and $\epsilon=3^{-56}$, and $G$ has girth at least $6$ (\cite{DMMP20}). 
 \end{itemize}
\end{thr}

%The result on $\epsilon-$flexibility for $7$-fold lists-assingments for planar graphs simply follows from $\epsilon-$flexibility for (degeneracy+2)-fold list-assignments.

Note that when demonstrating (weighted) $\epsilon$-flexibility in order to bound $\chi_{\text{wflex}}$, it is irrelevant how small $\epsilon$ is, as long as it is nonzero. One could argue that this is not very desirable, as ideally we want a large portion of our requests to be satisfied. Spoiler alert: in this work we will drive this to its logical conclusion by studying the ultimate form of weighted flexibility, captured in the \emph{fractional list packing number} $\chi_{\ell}^{\bullet}(G)$ (introduced in~\cite{CCDK23}), which requires $\epsilon$ to be as large as can be, namely of size $1/k$. We note that the connection was already observed in~\cite[Prop.~17]{KMMP22}.

We also study the \emph{fractional correspondence packing number} $\chi_{c}^{\bullet}(G)$, which is an upper bound for $\chi_{\ell}^{\bullet}(G)$. Moreover, as their names suggest, $\chi_{\ell}^{\star}(G)$ and $\chi_{c}^{\star}(G)$ are upper bounds for their fractional counterparts. See Section~\ref{sec:notation} for the formal definitions. \\ 

Using the framework of fractional packing, we recover multiple bounds from literature with short proofs. Not only do we cover a wide range of graph classes, in doing so we also obtain dramatically better (namely optimal) values for $\epsilon$, and we manage to generalize from list-colorings to correspondence-colorings. 
In particular, our result for girth $6$ graphs in~\cref{thm:fractcorpackplanar}
%significantly
strengthens the main result of Dvo\v{r}\'{a}k, Masa\v{r}\'{\i}k, Mus\'{\i}lek and Pangr\'{a}c \cite{DMMP20}, which in turn had answered a question of Dvo\v{r}\'{a}k, Norin and Postle~\cite{DNP19}.

\begin{thr}\label{thm:fractcorpackplanar}
For a planar graph $G$,
 $$ \chi_c^{\bullet}(G) \leq \begin{cases}
 5 & \text{ if } G \text{ is triangle-free}\\
 4 & \text{ if } G \text{ has girth } 5\\
 3 & \text{ if } G \text{ has girth at least } 6.
\end{cases}
$$
Moreover, this is optimal for the graphs with girth at least $6$. 
\end{thr}

Translating this back to flexibility, and specializing to list-assignments yields the next corollary, which makes the comparison with Theorem~\ref{thm:flexibility} more apparent.
\begin{cor}
 A planar graph $G$ is weighted $\epsilon$-flexible with respect to a list-assignment $L$ if 
 \begin{itemize}
 \item $L$ is a $5$-list-assignment, $\epsilon=1/5$, and $G$ is triangle-free.
 \item $L$ is a $4$-list-assignment, $\epsilon=1/4$, and $G$ has girth $5$. 
 \item $L$ is a $3$-list-assignment, $\epsilon=1/3$, and $G$ has girth at least $6$. 
 \end{itemize} 
\end{cor}

The main result of a very recent paper by Bi and Bradshaw~\cite{BB23} extended the third point of Theorem~\ref{thm:flexibility} in a different direction; they derived that in fact the class of graph with maximum average degree strictly smaller than $3$ is weighted $\epsilon-$flexible with respect to $3$-list-assignments, with $\epsilon=2^{-30}$. We essentially recover that result as well (see Theorem~\ref{thm:mad<3} for the exceptions), again with the optimal value $\epsilon=1/3$, and extended to correspondence-colorings.

More broadly, our upper bounds in Theorem~\ref{thm:listpacking} are corollaries of bounds in terms of the maximum average degree of $G$, denoted $\mad(G)$. 

\begin{thr}\label{thr:madsummary}
For a graph $G$,
$$ \chi_c^{\star}(G) \leq \begin{cases}
 8 & \text{ if } \mad(G)<6 \\
 5 & \text{ if } \mad(G)<4\\
 4 & \text{ if } \mad(G)<\frac{10}{3}\\
 2 & \text{ if } \mad(G)<2.
\end{cases}$$
\end{thr}

Multiple facets of~\cref{thr:madsummary} are best possible since every tree $T$ with at least one edge satisfies $\chi_c^{\star}(T)=2$, there exists no graph $G$ with $\chi_c^{\star}(G)=3$, every cycle $C$ satisfies $\chi_c^{\star}(C)=4$ while $\mad(C)=2$, and $\chi_c^{\star}(K_5)=6$ while $\mad(K_5)=4$.\\

We now wish to take a small excursion to other classes of structured graphs. Since planar graphs are precisely the graphs without $K_5$ and $K_{3,3}$ as a minor, we are more broadly interested in graphs with a small excluded minor. 
The $K_3$-minor-free graphs are forests, which are $1$-degenerate and hence (compare~\cite{CCDK23}) all packing numbers are at most $2$.

It is well-known~\cite[Thm.~17 \& Thm.~42]{Bod98} that a graph is $K_4$-minor-free if and only if it has treewidth at most $2$ if and only if each of its blocks is a series-parallel graph. Furthermore, a graph is outerplanar if and only if it is both $K_4$-minor-free and $K_{2,3}$-minor-free. As these graph classes are very structured, we were able to obtain a full characterization.

\begin{thr}\label{thm:excludedminorK4orK23}
Let $G$ be a graph that is either $K_4$-minor-free or $K_{2,3}$-minor-free. Then
$$\chi_{\ell}^{\star}(G),\chi_c^{\bullet}(G),\chi_c^{\star}(G) \leq 4 \text{ and } \chi_{\ell}^{\bullet}(G)\leq 3.$$
All of these inequalities are sharp and attained by some outerplanar graph.
\end{thr}

In the spirit of our considerations of planar graphs under girth constraints, we zoom in a bit more and also investigate the $K_4$-minor-free graphs with girth at least $4$.

\begin{thr}
Let $G$ be a $K_4$-minor-free graph with girth at least $4$. Then
$$\chi_c^{\star}(G) \leq 4 \text{ and } \chi_{\ell}^{\bullet}(G),\chi_{\ell}^{\star}(G),\chi_c^{\bullet}(G)\leq 3.$$
All of these inequalities are sharp and attained by cycles.
\end{thr}

%We only need a small extra argument to expand our bound on planar graphs from Theorem~\ref{thm:listpacking} to $K_5$-minor-free graphs. 

%\begin{cor}
% Let $G$ be a $K_5$-minor-free graph. Then $ \chi_{c}^{\star}(G) \leq 8$.
%\end{cor}

Leveraging a structural characterization due to~\cite{EMKT16}, we also study $K_{2,4}$-minor-free graphs. In this situation, the family of extremal graphs is governed by $K_5$ and $K_5^-$, the latter denoting $K_5$ minus an edge. The theorem below implies that $K_5$ is the only $2$-connected $K_{2,4}$-minor-free graph $G$ with $\chi_c^{\star}(G)=6$. The surprising result here is that while $6$ is a sharp upper bound, morally the upper bound for $\chi_c^\star$ for this class of graphs is $4.$

%Since $K_5$ is the only $K_{2,4}$-minor-free graph with $K_5$ as a subgraph, $K_5$ is the only $2$-connected $K_{2,4}$-minor-free graph $G$ with $\chi_c^{\star}(G)=6$.

\begin{thr}\label{thm:K24minorfree}
Let $G$ be a $K_{2,4}$-minor-free graph. Then
$\chi_c^{\star}(G)= 6$ if $K_5$ is a subgraph, $\chi_c^{\star}(G)=5$ if $K_5^-$ is a subgraph but $K_5$ is not, $\chi_c^{\star}(G)= 4$ if $G$ contains a cycle but not $K_5^-$, and $\chi_c^{\star}(G)= 2$ otherwise.
\end{thr}

In Table~\ref{table:overview} we summarize some bounds for several classes of (planar) graphs, and put them in direct comparison with common graph invariants such as the fractional chromatic number $\chi_f$, chromatic number $\chi$, list chromatic number $\chi_{\ell}$, correspondence chromatic number $\chi_c$ as well as all the packing numbers we are interested in. We note that these graph invariants are mostly linearly ordered: for every graph class:
$$\chi_f \leq \chi \leq \chi_{\ell} \leq \chi_{\text{wflex}} \leq \chi_{\ell}^{\bullet} \leq \chi_{\ell}^{\star} \leq \chi_c^{\star}.$$

 \begin{table}[ht]
 \centering
 \begin{tabular}{|c|c|c|c|c|c|c|c|}
 \hline
 &girth $3$ & girth $4$ & girth $5$ & girth $\ge 6$ & bipartite & subdivision & outerplanar \\
 \hline
 $\chi_{c}^\star$ & $[5,\bb{8}]$ & $[4,\bb{5}]$ & \bb{4} & $\bb{4}$ & $[4,\bb{5}]$ & \bb{4} & \bb{4}\\ 
 $\chi_{\ell}^\star$ & $[5,\bb{8}]$& $[4,\bb{5}]$ & \bb{4} & \bb{4} & $[3,\bb{5}]$ & \bb{3} & \bb{4}\\
 $\chi_{c}^\bullet$ & $[5,\bb{8}]$ & $[4,\bb{5}]$ & $[3,\bb{4}]$ & \bb{3} & $[4,\bb{5}]$ & \bb{3} & \bb{4}\\
 $\chi_{\ell}^\bullet$ & $[5,\bb{8}]$ & $[4,\bb{5}]$ & $[3,\bb{4}]$ & \bb{3} & $[3,\bb{5}]$ & \bb{3} & 3\\
 \hline
 \hline
 $\chi_{\text{wflex}}$ & $[5,7_{7^{-36}}]$ 
& $4_{2^{-186}}$ & $[3,4_{\bb{1/4}}]$ & $3_{\bb{1/3}}$ & $[3,4_{2^{-186}}]$& $3_{\bb{1/3}}$& $3_{1/3}$\\
 \hline
 \hline
 $\chi_{c} $ & 5& 4 & 3 & 3 & 4& 3 & 3 \\ 
 $\chi_{\ell} $ & 5& 4 & 3 & 3& 3 & 3 & 3\\ 
 $\chi $ & 4 & 3 & 3 & 3 & 2 & 2 & 3 \\ 
 $\chi_f$ & 4 & $3-o(1)$ & $\left[\frac{11}{4},3\right]$ & $2+\Theta \left( \frac 1g \right)$& 2& 2 & 3\\ 
 \hline
 \end{tabular}
 \caption{For various coloring parameters and various classes of planar graphs, the range in which the optimal upper bound must lie. See~\cite{ AH76,AT92,BK19,BMS22,DNP19,DSV15,G59,PU02,DMMP21,Thomassen94, Thomassen95,Voigt93,Voigt95}.
 For weighted $\epsilon-$flexibility, if $\chi_{\text{wflex}}(\G) \leq k$ is the best-known upper bound for the graph class $\G$ and some $\epsilon>0$, we write $k_{\epsilon}$ if $\epsilon$ is the largest value for which that upper bound $k$ is proved. The numbers in bold blue mark our improved bounds.}\label{table:overview}
\end{table}
% chi_c^{\bullet} \leq 4 is optimal for bipartite planar due to the the cube graph.
% chi_c<=3 for girth 5 is adaptation of thomassens list-proof, noted by Dvorak&Postle in paper introducing chi_c 
% $chi_f(G)<= 3-\epsilon$ for girth 4 = conjecture of Dvorak and Mnich
% $chi_c(G)>=4 possible for planar bipartite due to Bernshteyn+Kostochka 2017$~\cite{BK19}

\subsection{Methods and Outline}

In our study of the correspondence packing number, we use a careful analysis of derangements; permutations that avoid fixed points, i.e.,
a {\em derangement} of a permutation $a_1a_2\ldots a_k$ of $[k] $ is a permutation $b_1b_2\ldots b_k$ of $[k]$ such that $b_i \ne a_i$ for $i=1,2,\ldots, k$.

More specifically, we typically need to know how many common derangements of a given collection of permutations there are (given that there exists at least one), as this can be used to extend a partial packing on some neighbors of a vertex $v$ to a packing of $v$ in multiple ways, ensuring sufficient freedom to subsequently pack the remainder of the graph. We use e.g. the following theorem by Marshall Hall~\cite{Hall48}.
\begin{thr}\label{thr:MarshallHall}
 If $\Gamma=(A \cup B,E)$ is a bipartite graph such that $\abs{N(I)}\ge \abs{I}$ for every $I \subset A$ and the minimum degree of the vertices in $A$ is $d$, then there are at least $d!$ perfect matchings (saturating $A$).
\end{thr}

In a few situations we use a computer search to exhaustively explore the final set of possibilities.
The verifications by computer are gathered in~\cite{CD23}.

The main engine for our upper bounds on fractional packing numbers is our technical Lemma~\ref{lem:technicalinduction}. Its message: for a proof by induction on the number of vertices, it suffices to find an induced subgraph $T$ with fractional chromatic number at most one smaller than the target bound, such that each vertex of $T$ has at most one neighbor outside $T$, and such that each of those external neighbors has large enough list size. With discharging arguments (or directly) we then find such a $T$ and apply the lemma.
Critically, $T$ does not need to be a small subgraph. In fact, we will apply it with $T$ being a long path, a tree, or even the complement of a single vertex. Thus, in a sense this allows us to break away from local arguments.

%This is done in e.g.~\cref{lem:chi_c^star>4}, where we analyse $4$-correspondence-covers and show that in order to show that $\chi_c^{\star}(G)\leq 4$ for some graph, we may assume that it does not contain a path of three subsequent degree-$3$ vertices as a subgraph.\\

%A fact (see Theorem 9 in~\cite{CCDK21}) that we will often tacitly apply is that $\chi_c^{\star}(G)$ -and hence all the other packing numbers as well- is upper bounded by two times the \emph{degeneracy} of $G$, and hence upper bounded by twice the maximum degree.\\

We will often apply the fact that all packing numbers are upper bounded by two times the \emph{degeneracy}, and hence upper bounded by twice the maximum degree. More generally, we may assume that the graph under consideration does not have any vertex of small degree:
\begin{lem}\label{lem:degeneracy}
Let $G$ be a graph with a vertex $v$ of degree $d$ and let $k\geq 2d$ be such that $\chi_{c}^{\star}(G\setminus v) \leq k$. Then $\chi_{c}^{\star}(G) \leq k$. The same holds mutatis mutandis for $\chi_{c}^{\bullet},\chi_{\ell}^{\bullet}$ and $\chi_{\ell}^{\star}$.
\end{lem}

For a proof, see~\cite[Thm.~9]{CCDK21}. The adaptation for the fractional versions can be done by applying Hall's marriage theorem, taking into account the equivalent definition of the fractional packing number with a set of $mk$ colorings.

In~\cref{sec:notation}, the reader can find more information about all terminology and notation.
%The main proofs start in~\cref{sec:TechLemmasforFracPacking}.
A key lemma used for bounding the correspondence packing number of planar graphs with large girth is explained in~\cref{sec:TechLemmasforPacking}. With a tight construction and some discharging, this leads to the determination of the (integral) packing numbers of planar graphs with girth at least $5$, in~\cref{sec:example_atleast4}. After that, we consider triangle-free planar graphs in~\cref{sec:packing_trianglefreeplanar} and planar graphs without additional constraints in~\cref{sec:planargraphs}.
\cref{sec:TechLemmasforFracPacking} contains the main lemmas for fractional packing, consequently applied in~\cref{sec:varia} and~\cref{sec:planargirth6}, where it is proven that $\chi_c^{\bullet} \le 3$ for planar graphs with girth at least $6.$
In~\cref{sec:K4minorfree}, we prove sharp upper bounds on the packing numbers of $K_4$-minor-free graphs.
Here we use induction in some creative ways. The stronger induction hypothesises make the proofs easier.
In~\cref{sec:K24minorfree}, we prove the behavior of $K_{2,4}$-minor-free graphs.
Finally, we give some remarks and directions for further studies in~\cref{sec:conc}.

\section{Notation and definitions}~\label{sec:notation}
Let $G$ be a graph. 
A pair $\sH=(L,H)$ is a \emph{correspondence-cover} of a graph $G$ if $H$ is a graph and $L:V(G)\to 2^{V(H)}$ is a mapping that satisfies the following:
\begin{enumerate}[(i)]
 \item $L$ induces a partition of $V(H)$,
 \item the bipartite subgraph of $H$ induced between $L(u)$ and $L(v)$ is empty whenever $uv\notin E(G)$,
 \item\label{itm:corrdef} the bipartite subgraph of $H$ induced between $L(u)$ and $L(v)$ is a matching $M_{uv}$ whenever $uv\in E(G)$, 
 \item the subgraph of $H$ induced by $L(v)$ is a clique for each $v\in V(G)$.
\end{enumerate}
It can be convenient to drop $\sH$ and $L$ from the notation, saying for example that some property of the correspondence-cover holds if it holds for $H$ as a graph.
A correspondence-cover is \emph{$k$-fold} if $|L(v)|=k$ for each vertex $v$ of $G$.

A list-assignment $L$ of $G$ naturally gives rise to a correspondence-cover $(\tilde L, H)$ of $G$ by defining a vertex $x_v$ for every $v\in V(G), x \in L(v)$, setting $\tilde L(v) = \{x_v \mid x\in L(v) \}$, and forming $H$ on these vertices by adding a clique on $\tilde L(v)$ for every $v\in V(G)$, and adding edges of the form $x_ux_v$ for each color $x\in\bigcup_{v\in V(G)}L(v)$ and edge $uv\in E(G)$. 
We call a correspondence-cover that arises from a list-assignment in this way a \emph{list-cover} of $G$.

With a slight abuse of terminology, we will sometimes also refer to $L(v)$ as a `list', even if $(L,H)$ is a correspondence-cover that does not necessarily arise from a list-assignment.

An \emph{independent transversal} of a cover $\sH$ is an independent set of $H$ that contains precisely one vertex of $L(v)$, for every vertex $v$ of $G$.
It is useful to keep in mind that the independent transversals of $\sH$ are (in bijection with) correspondence-colorings of $G$. Furthermore, if $\sH$ arises from a list-assignment $L$ then its independent transversals are (in bijection with) the $L$-colorings of $G$. 
We say that a cover $\sH$ has a \emph{packing} if it is $k$-fold for some $k\geq 1$, and it has $k$ vertex-disjoint independent transversals.
We say that a cover $\sH=(L,H)$ of a graph $G$ has a \emph{fractional packing} if there exists a probability distribution on independent transversals $I$ of $\sH$, such that for every vertex $v$ of $G$ and every vertex $x\in L(v)$ of $H$, we have $\mathbb{P}(x \in I)= 1 / |L(v)|$. 
Note that if $\sH$ has a packing, then it also has a fractional packing.

It is also possible to view packing through the lense of hypergraphs. Consider the hypergraph with vertex set $V(H)\cong\prod_{v\in V(G)}L(v)$, and hyperedges given by the $L$-colorings. Then this hypergraph has a perfect matching if and only if $(L,H)$ admits a packing. 

As detailed in~\cite{CCDK23}, (fractional) colorings of the cover graph form an alternative perspective on (fractional) packings of $k$-fold covers. Indeed, a $k$-fold cover $(H,L)$ has a packing if and only if the chromatic number of $H$ is $k$, and it has a fractional packing if and only if the fractional chromatic number of $H$ is $k$. However, we stress that for $\sH$ to have a packing it is necessary that all lists have the same size, while for a fractional packing this is not required. In that sense the notion of fractional packing is more versatile.
In this work, we focus our study on four coloring parameters:
\begin{itemize}
 \item The \emph{correspondence packing number} $\chi_{c}^{\star}(G)$ is the minimum integer $k\geq 1$ such that every $k$-fold correspondence-cover of $G$ has a packing.
 \item The \emph{fractional correspondence packing number} $\chi_{c}^{\bullet}(G)$ is the minimum integer $k\geq 1$ such that every $k$-fold correspondence-cover of $G$ has a fractional packing.
 \item The \emph{list packing number} $\chi_{c}^{\star}(G)$ is the minimum integer $k\geq 1$ such that every $k$-fold list-cover of $G$ has a packing.
 \item The \emph{fractional list packing number} $\chi_{\ell}^{\bullet}(G)$ is the minimum integer $k\geq 1$ such that every $k$-fold list-cover of $G$ has a fractional packing.
\end{itemize}

Given a graph $G$ with cover $(L,H)$ 
 and a subgraph $G_0$ of $G$, the \emph{restriction of $(L,H)$ to $G_0$} is the cover $(L_0,H_0)$ of $G_0$ with $L_0(v)=L(v)$ for every $v \in V(G_0)$, and with $H_0$ obtained from $H$ by removing the vertices $L(u)$ with $u\in V(G)\setminus V(G_0)$, and removing the matching between $L(u)$ and $L(v)$ for every $uv\in E(G)\setminus E(G_0)$.

A $k$-fold correspondence-cover $(L,H)$ of $G$ is called \emph{full} if for every $uv \in E(G)$, the matching between $L(u)$ and $L(v)$ is a perfect matching.
Moreover, a $k$-fold cover $(L,H)$ of $G$ has \emph{full identity matchings} if there exist labelings $L(v)=\{1_v,\ldots, k_v\}$ of the lists, such that for every $uv \in E(G)$, the matching $M_{uv}$ between $L(u)$ and $L(v)$ satisfies $i_uj_v \in M_{uv}$ if and only if $i=j$. So in this case $H$ is isomorphic to the Cartesian product of $K_k$ and $G$. Given any induced subgraph $T$ of $G$ which is isomorphic to a forest, we may use the absence of cycles to `untwist' the matchings, thus allowing us to assume without loss of generality that the correspondence-cover restricted to $T$ has full identity matchings. This fact allows us to lighten notation in some of our arguments. 

Where full identity matchings are not applicable, we will need explicit labelings to describe the correspondence-cover. When working with an explicit $k$-fold cover, we will often label each list $L(v)$ as $\{1_v, 2_v, \ldots, k_v\}$ endowed with a natural order that we assume is clear, and it can be convenient to omit the
subscripts, writing $L(v)=\{1,2,\ldots, k\}$ for short. 
%and usually we write a permutation of the set as an ordered sequence of comma-separated values, such as f = (2, 1, 3) for the permutation f with f(1) = 2, f(2) = 1, f(3) = 3. 
In these situations, we denote a correspondence-packing by $\vec{c}=(c_1,\ldots, c_k)$, where $c_1,\ldots, c_k$ are disjoint correspondence-colorings, and $\vec{c}(v)=(c_1(v),c_2(v), \ldots, c_k(v))$. Thus for instance, for a list $L(v)=\{1_v,2_v,3_v,4_v\}$, the vector $\vec{c}(v)=(2,1,4,3)$ indicates that $2_v$ occurs in the first coloring, $1_v$ occurs in the second coloring, $4_v$ in the third coloring and $3_v$ in the fourth. \\
A partial packing of $G=(V,E)$ is a mapping $\vec{c}$, which assigns to each vertex $v$ of a set $U \subset V$ a permutation $\vec{c}(v) = (c_1(v), \ldots, c_k(v))$ of $[k]$. 
For an edge $e=uv$, we say $\vec{c}(u)$ and $\vec{c}(v)$ are {\em compatible} with respect to $(L,H)$ if $c_i(u)c_i(v) \notin M_{uv}$ for $i=1,2,\ldots, k$. 
When $M_{uv}$ is a perfect matching, this is closely related with the notion of a derangement.
Note that $\vec{c}(u)$ and $\vec{c}(v)$ are compatible with respect to $M_{uv}$ precisely when $M_{uv} \vec{c}(u)$ (considered as composition of two permutations of $[k]$) is a derangement of $\vec{c}(v).$

Given a subgraph $T$ of some other graph $G$, we write $G\setminus T$ for the graph induced by the vertices of $G$ that are not in $T$.

Finally: the \emph{maximum average degree} of a graph $G$, denoted $\mad(G)$, is the maximum of the average degrees of its subgraphs.

\section{A technical lemma for integral packing}\label{sec:TechLemmasforPacking}

In this section, we prove a technical Lemma that will allow us to assume that certain minimum counterexamples do not contain a cycle or path of length $3$ as induced subgraph. It is a strengthening of~\cite[Thr.~15]{CCDK23}, where a similar method was used to show $\chi_c^{\star}(G)\leq 4$ for every subcubic graph.

\begin{lem}\label{lem:chi_c^star>4}
 Let $G$ be a graph that contains an induced subgraph $T$ which is isomorphic to $P_3$ or $C_3$, all of whose vertices have degree $3$ in $G$. 
 If $G\setminus T$ has correspondence packing number at most $4$, then also $\chi_c^{\star}(G) \leq 4$.
\end{lem}
\begin{proof}
 Let $G=(V,E)$ and $\sH=(L,H)$ be a full $4$-fold correspondence-cover of $G$.
 Let the vertices of $T$ have vertices $u,v,w$.
 By~\cref{lem:degeneracy}, the result is immediate if one of them has degree at most $2.$ So we assume that $\deg(u)=\deg(v)=\deg(w)=3.$
 There are two cases to consider. These are depicted in~\cref{fig:localpartC3P3}.
 
 \textbf{Case 1: $T=C_3$}
 
 Let the third neighbor of $u,v,w$ be the (not necessarily different) vertices $u_1, v_1, w_1$ respectively.
 Take a partial packing $\vec c$ of $G \setminus T$. We assume without loss of generality that $\vec{c}(u_1)=(1,2,3,4).$
 Since we will check all ($4!)^2$ potential partial packings on the remaining neighbors $v_1,w_1$, we may actually assume that $u_1,v_1,w_1$ are distinct. As the subgraph $H$ formed by $\{uv,uw,uu_1,vv_1,ww_1\}$ is a tree, we can assume that the cover restricted to $H$ has full identity matchings. The only edge that (possibly) does not respect the identity matchings is $vw$. We let $s$ be the matching (permutation of $[4]$) on $vw.$
 There are $5$ types of permutations on $[4],$ which we will describe with a disjoint cycle decomposition.
 One identity, $6$ permutations switching $2$ elements (e.g. $(12)$), $8$ rotations on $3$ elements (e.g. $(123)$), $3$ double switches (e.g. $(12)(34)$) and $6$ rotations on $4$ elements (e.g. $(1234)$). As such, to check all possible full matchings for $s$, it is sufficient to consider one matching of each of the five types.
 With a brute force check, we can verify that for each of the $5 \cdot 24^2$ combinations of $s, \vec{c}(v_1), \vec{c}(w_1),$ there is at least one choice of $\vec{c}(w)$ for which the partial packing can also be extended with a valid $\vec{c}(u)$ and $\vec{c}(v),$ i.e. in all cases a partial packing $\vec c$ of $G \backslash T$ can be extended to a packing of $G.$
 See~\cite[\text{Verifications\_NoInducedC3P3}]{CD23} for the elementary brute-force verification. 
 
\begin{figure}[h]
 \centering
 \begin{tikzpicture}
 \foreach \x in {0,2}
 {
 \draw[fill] (\x,0) circle (0.1);
 }
 \draw[thick] (0,0)--(2,0)--(60:2)--(0,0);
 \draw[fill] (60:2) circle (0.1);

\draw[fill] (-1,0) circle (0.1);
\draw[fill] (3,0) circle (0.1);
 \draw[dotted] (-1, 0)--(0,0);
\draw[dotted] (3, 0)--(2,0);
 
 \draw[fill] (2,1.73205) circle (0.1);
 \draw[dotted] (60:2)--(2,1.73205);

 \node at (-1.65,0) {$\vec{c}(w_1)$};
 \node at (3.65,0) {$\vec{c}(u_1)$};
 \node at (2.6,1.73205) {$\vec{c}(v_1)$};

 \node at (0.25,1) {$s$};
 \node at (0.1,-0.3) {$w$};
 \node at (2,-0.3) {$u$};
 \node at (1,2.05) {$v$};

 \end{tikzpicture}\quad
 \begin{tikzpicture}
 \foreach \x in {0,2,4}
 {
 \draw[fill] (\x,0) circle (0.1);
 }
 \draw[thick] (0,0)--(4,0);

 \foreach \y in {-1,1}{
\draw[fill] (-1,\y) circle (0.1);
\draw[fill] (5,\y) circle (0.1);
 \draw[dotted] (-1, \y)--(0,0);
\draw[dotted] (5, \y)--(4,0);
 }
 \draw[fill] (2,1) circle (0.1);
 \draw[dotted] (2,1)--(2,0);

 \node at (-1.65,1) {$\vec{c}(w_1)$};
 \node at (-1.65,-1) {$\vec{c}(w_2)$};
 \node at (5.65,1) {$\vec{c}(v_1)$};
 \node at (5.65,-1) {$\vec{c}(v_2)$};
 \node at (2.6,1) {$\vec{c}(u_1)$};

 \node at (0.1,-0.3) {$w$};
 \node at (2,-0.3) {$u$};
 \node at (3.9,-0.3) {$v$};

 \end{tikzpicture}
 \caption{The local structures (induced $C_3$ and $P_3$ resp.) of $G$ studied in~\cref{lem:chi_c^star>4}}\label{fig:localpartC3P3}
\end{figure}
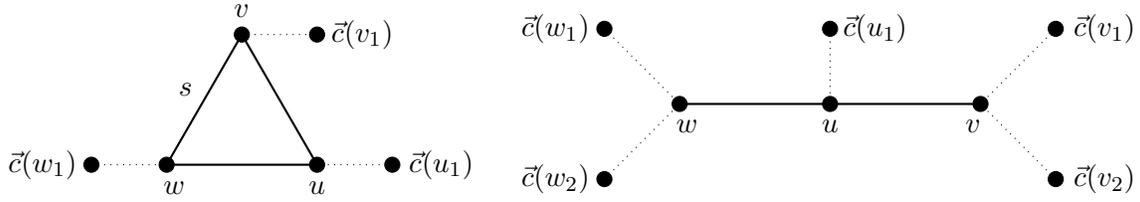

 \textbf{Case 2: $T=P_3$}

 Let $wuv$ be the induced path $P_3$, with the center $u$ having $u_1$ as its third neighbor.
 Let the neighbors of $v,w$ different from $u$ be the (not necessarily different) vertices $v_1, v_2, w_1$ and $w_2$ respectively. 
 Take a partial packing $\vec c$ of $G \setminus T$.
 Up to relabeling, we can assume that all matchings considered are full identity matchings and $\vec{c}(u_1)=(1,2,3,4).$
 Indeed, since we will check for all $(4!)^4$ potential partial packings on $\{v_1,v_2,w_1,w_2\}$ we may assume that $u_1,v_1,v_2,w_1,w_2$ are distinct, so that the graph $H$ on edges $\{ww_1,ww_2,wu,uu_1,uv,vv_1,vv_2 \}$ is a tree, allowing us to assume that the cover restricted to $H$ has full identity matchings.
 Again a brute-force verification (done in~\cite[\text{Verifications\_NoInducedC3P3}]{CD23}) shows that for every choice of $\vec{c}(v_1), \vec{c}(v_2), \vec{c}(w_1)$ and $\vec{c}(w_2)$ the partial packing can be extended to a proper partial packing of $G.$ 
\end{proof}

\begin{rem}
 The above proof for correspondence-packing also shows the analogous statement for list-packing, i.e. $\chi_{\ell}^\star(G) \le 4$ if $\chi_{\ell}^\star(G\setminus T) \le 4$.
 %, regardless whether $\chi_{c}^\star(G\setminus T)> 4$ or not.
In the case of list-packing, $T$ is also allowed to be a $C_3$ whose vertices have degree $3,3$ and $4$ respectively.
\end{rem}

\section{Packing numbers of planar graphs with girth $\geq 5$}\label{sec:example_atleast4}

In this section we determine the optimal values of $\chi_{\ell}^{\star}$ and $\chi_c^{\star}$ for planar graphs of girth $\geq 5$.
\begin{thr}\label{thr:main_largegirth_no3packing}
For every $g\geq 3$, there exists a planar graph $G$ with girth at least $g$ such that $\chi_{\ell}^{\star}(G)= 4$.
\end{thr}

\begin{proof}
 Consider $3$ vertex-disjoint odd cycles $A,B,C$ of length $n=2k+1 \ge g$, with $V(A)=\{a_1, a_2, \ldots, a_n\}, V(B)=\{b_1, b_2, \ldots, b_n\}$ and $V(C)=\{c_1, c_2, \ldots, c_n\}.$\\
 Assign the lists $\{1,2,3\},\{1,2,4\},\{1,3,4\}$ to the vertices of these cycles such that:
 \begin{itemize}\itemsep0em
 \item for odd $i$, $L(a_i)=\{1,2,3\}, L(b_i)=\{1,2,4\}, L(c_i)= \{1,3,4\}$,
 \item for even $i,$ $L(a_i)=\{1,2,4\}, L(b_i)= \{1,3,4\}$ and $L(c_i)=\{1,2,3\}$. 
 \end{itemize}
 For $1 \le i \le n-2$, connect $a_i$ with $a_{i+2}$ with a path of length at least $g$, such that this path is at the outside of the cycle when $i$ is even and inside the cycle if $i$ is odd. 
 Do the same for the cycles $B$ and $C$.
 Finally add paths of length at least $g$ between $a_1$ and $c_2$, between $b_1$ and $a_2$, and between $c_1$ and $b_2.$
 All vertices on the added paths are assigned the same list as the end vertices of the path (equivalently, the nearest vertex of degree at least $3$).

 This final graph $G$ with list-assignment $L$ is our construction. It is immediate that $G$ has girth at least $g$, since the three initial cycles have and every other cycle would use at least one path that has length at least $g$. 
 %One can even replace most paths by length at least $g-2$ and the paths between the different cycles by length at least $\frac g3.$
 We now prove that there exists no proper $L$-packing $\vec c$.
 
Assume to the contrary that $\vec c$ is a proper $L$-packing of $G$. Note that for each of the lists $\{1,2,3\},\{1,2,4\}$ and $\{1,3,4\}$, all vertices which are assigned the same list induce a connected graph. As such, all of them need to be assigned with a cyclic permutation of a fixed permutatation of the list.
 
 We may assume that the first entry in all the three fixed lists (as a representant for the even permutations of it) is $1$. By the pigeon hole principle there is at least one element in $\{2,3,4\}$ which occurs once as the second and once as the third entry.
 Without loss of generality, this is the case for the element $2$ and we can assume that $\vec c$ assigns even permutations of $(1,2,3)$ and $(1,4,2),$ and that $\vec c(a_1)=(1,2,3).$
 In cycle $A$, this immediately implies that $\vec c(a_2)=(2,1,4)$ and by induction $\vec c(a_i)=(1,2,3)$ for odd $i$ and $\vec c(a_i)=(2,1,4)$ for even $i$. 
 But since $\vec c(a_1)$ and $\vec c(a_n)$ are equal, this is not a proper packing.
 An example for $g=5$ has been presentented in Figure~\ref{fig:ex_forg5}. Here the vertices have the same color if they are assigned the same list. 
\end{proof}

\begin{figure}[h]
 \centering
 \definecolor{uuuuuu}{rgb}{0.26666666666666666,0.26666666666666666,0.26666666666666666}
\definecolor{ffwwqq}{rgb}{1,0.4,0}
\definecolor{wwqqcc}{rgb}{0.4,0,0.8}
\begin{tikzpicture}[line cap=round,line join=round,>=triangle 45,x=1cm,y=1cm]
%\clip(13.285975793681956,-15.901197452455522) rectangle (40.1863409476574,-1.5488684831089234);
\draw [line width=2pt] (21.85140368991604,-5.124864394953209)-- (20.467702930727047,-6.1301818425373185);
\draw [line width=2pt] (20.467702930727047,-6.1301818425373185)-- (20.996229590478208,-7.756819642211699);
\draw [line width=2pt] (20.996229590478208,-7.756819642211699)-- (22.706577789353872,-7.756819642211699);
\draw [line width=2pt] (22.706577789353872,-7.756819642211699)-- (23.235104449105034,-6.1301818425373185);
\draw [line width=2pt] (23.235104449105034,-6.1301818425373185)-- (21.85140368991604,-5.124864394953209);
\draw [line width=2pt] (17.85230687417744,-12.566255981964733)-- (19.414787702247505,-13.261917266866162);
\draw [line width=2pt] (19.414787702247505,-13.261917266866162)-- (20.55923402964596,-11.990880853107132);
\draw [line width=2pt] (20.55923402964596,-11.990880853107132)-- (19.70405993020813,-10.509675863563846);
\draw [line width=2pt] (19.70405993020813,-10.509675863563846)-- (18.03108694305851,-10.865277249479195);
\draw [line width=2pt] (18.03108694305851,-10.865277249479195)-- (17.85230687417744,-12.566255981964733);
\draw [line width=2pt] (26.29628943590652,-12.308879623082058)-- (26.11750936702545,-10.607900890596522);
\draw [line width=2pt] (26.11750936702545,-10.607900890596522)-- (24.44453637987583,-10.25229950468117);
\draw [line width=2pt] (24.44453637987583,-10.25229950468117)-- (23.589362280437996,-11.733504494224455);
\draw [line width=2pt] (23.589362280437996,-11.733504494224455)-- (24.73380860783645,-13.004540907983486);
\draw [line width=2pt] (24.73380860783645,-13.004540907983486)-- (26.29628943590652,-12.308879623082058);
\draw [line width=2pt] (20.996229590478208,-7.756819642211699)-- (19.70405993020813,-10.509675863563846);
\draw [line width=2pt] (20.55923402964596,-11.990880853107132)-- (23.589362280437996,-11.733504494224455);
\draw [line width=2pt] (22.706577789353872,-7.756819642211699)-- (24.44453637987583,-10.25229950468117);
\draw [shift={(24.942825858172256,-12.021192058653256)},line width=2pt] plot[domain=2.9321531433504737:6.073745796940267,variable=\t]({1*1.3837007591889943*cos(\t r)+0*1.3837007591889943*sin(\t r)},{0*1.3837007591889943*cos(\t r)+1*1.3837007591889943*sin(\t r)});
\draw [shift={(18.778183402192784,-11.53796592276429)},line width=2pt] plot[domain=0.8377580409572787:3.9793506945470716,variable=\t]({1*1.3837007591889947*cos(\t r)+0*1.3837007591889947*sin(\t r)},{0*1.3837007591889947*cos(\t r)+1*1.3837007591889947*sin(\t r)});
\draw [shift={(22.278990739634956,-6.440842018582455)},line width=2pt] plot[domain=-1.2566370614359172:1.8849555921538759,variable=\t]({1*1.3837007591889945*cos(\t r)+0*1.3837007591889945*sin(\t r)},{0*1.3837007591889945*cos(\t r)+1*1.3837007591889945*sin(\t r)});
\draw [line width=2pt] (20.55923402964596,-11.990880853107132)-- (18.03108694305851,-10.865277249479195);
\draw [line width=2pt] (18.03108694305851,-10.865277249479195)-- (19.414787702247505,-13.261917266866162);
\draw [line width=2pt] (20.996229590478208,-7.756819642211699)-- (23.235104449105034,-6.1301818425373185);
\draw [line width=2pt] (23.235104449105034,-6.1301818425373185)-- (20.467702930727047,-6.1301818425373185);
\draw [line width=2pt] (24.44453637987583,-10.25229950468117)-- (24.73380860783645,-13.004540907983486);
\draw [line width=2pt] (24.73380860783645,-13.004540907983486)-- (26.11750936702545,-10.607900890596522);
\begin{scriptsize}
\draw [fill=wwqqcc] (20.996229590478208,-7.756819642211699) circle (4pt);
\draw [fill=black] (22.706577789353872,-7.756819642211699) circle (4pt);
\draw [fill=wwqqcc] (23.235104449105034,-6.1301818425373185) circle (4pt);
\draw [fill=black] (21.85140368991604,-5.124864394953209) circle (4pt);
\draw [fill=wwqqcc] (20.467702930727047,-6.1301818425373185) circle (4pt);
\draw [fill=ffwwqq] (19.414787702247505,-13.261917266866162) circle (4pt);
\draw [fill=ffwwqq] (20.55923402964596,-11.990880853107132) circle (4pt);
\draw [fill=wwqqcc] (19.70405993020813,-10.509675863563846) circle (4pt);
\draw [fill=ffwwqq] (18.03108694305851,-10.865277249479195) circle (4pt);
\draw [fill=wwqqcc] (17.85230687417744,-12.566255981964733) circle (4pt);
\draw [fill=ffwwqq] (26.29628943590652,-12.308879623082058) circle (4pt);
\draw [fill=black] (26.11750936702545,-10.607900890596522) circle (4pt);
\draw [fill=black] (24.44453637987583,-10.25229950468117) circle (4pt);
\draw [fill=ffwwqq] (23.589362280437996,-11.733504494224455) circle (4pt);
\draw [fill=black] (24.73380860783645,-13.004540907983486) circle (4pt);
\draw [fill=uuuuuu] (24.73380860783645,-13.004540907983486) circle (4pt);
\draw [fill=ffwwqq] (26.29628943590652,-12.308879623082058) circle (4pt);
\draw [fill=wwqqcc] (18.443466348655353,-10.195359396245287) circle (4pt);
\draw [fill=wwqqcc] (17.406384935835533,-11.356867168820218) circle (4pt);
\draw [fill=ffwwqq] (25.65059929579457,-13.210176731378892) circle (4pt);
\draw [fill=ffwwqq] (23.922121478222582,-12.955424556338787) circle (4pt);
\draw [fill=black] (23.592367883050034,-6.876351226447283) circle (4pt);
\draw [fill=black] (23.10230280533013,-5.328733645508614) circle (4pt);
\draw [fill=wwqqcc] (20.35014476034317,-9.133247752887772) circle (4pt);
\draw [fill=black] (23.57555708461485,-9.004559573446436) circle (4pt);
\draw [fill=ffwwqq] (22.07429815504198,-11.862192673665794) circle (4pt);
\draw [fill=ffwwqq] (18.72293732265301,-12.063597258172678) circle (4pt);
\draw [fill=ffwwqq] (19.295160486352238,-11.428079051293164) circle (4pt);
\draw [fill=black] (24.58917249385614,-11.62842020633233) circle (4pt);
\draw [fill=black] (25.42565898743095,-11.806220899290004) circle (4pt);
\draw [fill=wwqqcc] (22.115667019791623,-6.943500742374509) circle (4pt);
\draw [fill=ffwwqq] (18.663123714705375,-11.14667815038618) circle (4pt);
\draw [fill=ffwwqq] (19.927197257999097,-11.709479952200148) circle (4pt);
\draw [fill=ffwwqq] (19.06886251245026,-12.66275726251942) circle (4pt);
\draw [fill=ffwwqq] (18.37701213285576,-11.464437253825936) circle (4pt);
\draw [fill=black] (24.516854436865984,-10.940359855506749) circle (4pt);
\draw [fill=black] (24.661490550846295,-12.316480557157908) circle (4pt);
\draw [fill=black] (25.0797337976337,-12.405380903636745) circle (4pt);
\draw [fill=black] (25.7715841772282,-11.207060894943263) circle (4pt);
\draw [fill=wwqqcc] (21.555948305134915,-7.350160192293104) circle (4pt);
\draw [fill=wwqqcc] (22.675385734448327,-6.536841292455914) circle (4pt);
\draw [fill=wwqqcc] (21.85140368991604,-6.1301818425373185) circle (4pt);
\draw [fill=wwqqcc] (21.15955331032154,-6.1301818425373185) circle (4pt);
\draw [fill=wwqqcc] (22.54325406951054,-6.1301818425373185) circle (4pt);
\end{scriptsize}
\end{tikzpicture}
 \caption{A planar graph with girth $5$ which is not $3$-list packable.}
 \label{fig:ex_forg5}
\end{figure}
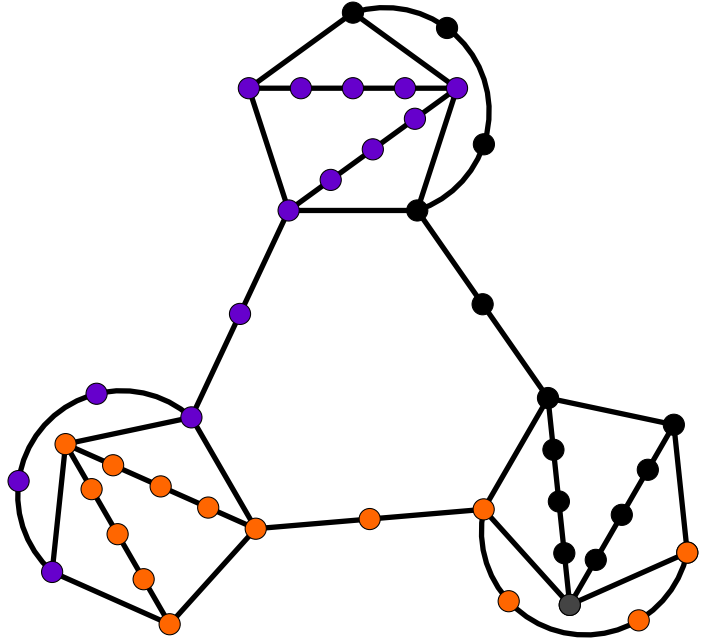

\begin{thr}\label{thm:mad<3+1/3}
 If $\mad(G)<\frac{10}{3},$ then $\chi_c^\star(G)\le 4.$
\end{thr}
\begin{proof}
 Assume this is not true and $G$ is a vertex-minimum counterexample.
 By Lemma~\ref{lem:degeneracy} the minimum degree of $G$ is $3$. 
 We prove by discharging that $G$ contains a degree-$3$ vertex with two degree-$3$ neighbors. Removing those three vertices yields another graph with maximum average degree $<\frac{10}{3}$ and correspondence packing number at most $4$, which by ~\cref{lem:chi_c^star>4} then implies that $\chi_c^{\star}(G)\leq 4$, contradiction. 
 So suppose now for a contradiction that every degree-$3$ vertex has at most one neighbor of degree $3$. Assign initial charge $\omega(v)=\deg(v)$ to each vertex $v$, and discharge as follows: every vertex with degree $\geq 4$ gives charge $\frac{1}{6}$ to each neighbor of degree $3$. If $\deg(v)\geq 4$, then the new charge of $v$ is at least $\frac{5}{6}\deg(v) \geq \frac{10}{3}$. If $\deg(v)=3$ then $v$ receives charge $\frac{1}{6}$ from at least two neighbors of degree $\geq 4$. Since every new charge is at least $\frac{10}{3}$ and the sum of the charges is preserved, the average degree of $G$ is at least $\frac{10}{3}$, contradiction. 
\end{proof}

Do there exist graphs $G$ with $\mad(G)=\frac{10}{3}$ and $\chi_c^{\star}(G)>4$? We do not know, but because $\chi_c^{\star}(K_5^{-})=5$ (see~\cref{lem:K_5^-}), it is not possible to weaken the condition in~\cref{thm:mad<3+1/3} to $\mad(G)\leq 18/5$. 
 We also note that~\cref{thm:mad<3+1/3} strengthens~\cite[Thr.~15]{CCDK23}, which only deals with graphs of maximum degree $3$. Moreover:

\begin{cor}
For every planar graph $G$ with girth at least $5$, $\chi_\ell^\star(G) \le \chi_c^\star(G)\le 4.$
\end{cor}
\begin{proof}
By Euler's formula, every planar graph with girth at least five has maximum average degree strictly smaller than $\frac{10}{3}$, so the conclusion follows from~\cref{thm:mad<3+1/3}.
\end{proof}

We remark that for Halin graphs, a subclass of planar graphs with degeneracy $3$, the packing numbers are also bounded by $4$ by essentially the same proof. A \emph{Halin graph} is a planar graph that is obtained from a series-reduced tree (a tree without any degree-$2$ vertex), by adding edges such that the graph induced by its leaves becomes a cycle. It is well-known that Halin graphs are $3$-connected and contain cycles of all but at most one length in $[3,|V(G)|].$
% We find that $\chi_{c}^{\star}(G)\leq 4$ for every Halin graph $G$. Moreover, there exist Halin graphs $G$ with $\chi_{\ell}^{\bullet}(G)=\chi_{c}^{\bullet}(G)=\chi_{\ell}^{\star}(G)=\chi_{c}^{\star}(G)=4$.

\begin{prop}
\label{halin}
 For a Halin graph $G$, we have that $\chi_c^\star(G) \le 4.$
% Furthermore, there exist Halin graphs for which $\chi_\ell^\bullet(G) = 4.$
\end{prop}

\begin{proof}
%For the upper bound, 
Notice that a Halin graph is either $K_4$ or the outer cycle has length at least $4$. Since all vertices on the outer cycle have degree $3$, we are done by~\cref{lem:chi_c^star>4}.
\end{proof}

The bound in Proposition \ref{halin} is sharp (also for the other packing numbers) by e.g. the Halin graph $G$ obtained from the balanced double broom of order $n=8$, which corresponds with the $K_4^-$-necklace plus an additional edge. See~\cite[Fig.~7]{CCDK23}. For this graph, $\chi_{\ell}^\bullet(G)=4.$
By the main theorem of~\cite{ZQZ21}, this gives an example for the following remark.

\begin{rem}
 There are graphs $G$ for which the Alon-Tarsi number is smaller than the fractional list packing number, i.e., $AT(G)<\chi_{\ell}^\star(G)$.
\end{rem}

While Theorem~\ref{thr:main_largegirth_no3packing} revealed the existence of large-girth graphs with a $3$-list-assignment $L$ that does not admit three disjoint $L$-colorings, we finish this section with the observation that \emph{two} disjoint $L$-colorings do always exist, even when generalized to correspondence-colorings:

\begin{thr}\label{thr:2disjcol_g8}
 Let $G$ be a planar graph with girth at least $8.$ 
 Then every $3$-correspondence-assignment of $G$ admits two disjoint proper correspondence-colorings.
 %Then every $3$-correspondence-cover of $G$ has at least two disjoint independent transversals.
 % Then every $3$-correspondence-cover of $G$ has at least two disjoint maximum independent sets. I.e. for every $3$-correspondence-assignment of the vertices of $G,$ we can find two disjoint proper correspondence-colorings.
\end{thr}

\begin{proof}
 Let $G$ be a minimum counterexample, for some $3$-correspondence-assignment.
 Note that by~\cite[Thr.~12(iii)]{JMMS16}, there is either a path $P_2$ whose vertices have degree $2$ in $G$ or a path $P_3$ whose vertices $u,v,w$ have degree $2,3$ and $2$ resp.
 Let $\vec c=(c_1,c_2)$ be the two partial correspondence-colorings of $G \setminus P_k.$ 
 The first case will follow from the second case, so we focus on that one.
 Let $u', v', w'$ be the remaining neighbors of $u,v,w$ (which are different due to the girth condition).
 We can assume that all matchings in the correspondence-cover of $G$ are full identity matchings, when restricting to the edges on the tree induced by $\{u,u',v,v',w,w'\}.$
 Note that $\vec{c}(v')$ forbids $3$ choices for $\vec{c}(v).$ 
 The coloring is not extendable in $u$ if and only if $c_1(u')=c_2(v)$ and $c_2(u')=c_1(v).$
 This implies that there is exactly one choice of $\vec c(v)$ that makes extending the 2 colorings properly in $u$ impossible.
 Symmetrically, there is exactly one choice of $\vec c(v)$ that would forbid extension to $w$.
 As such, there is at least one choice of $\vec c(v)$, which does not result in a conflict with $\vec c(v')$ and for which valid choices of $\vec c(u), \vec c(w)$ exist as well.
 This contradiction implies that $G$ was not a counterexample and thus no counterexample exists.
\end{proof}

\section{Packing numbers of triangle-free planar graphs}\label{sec:packing_trianglefreeplanar}

In this section we show that triangle-free planar graphs have correspondence packing number at most $5$. We actually derive this bound for all graphs with maximum average degree strictly smaller than $4$. The proof will make use of the following little lemmas. The first one is proven by discharging.

\begin{lem}\label{lem:subconf_K3freeplanar}
 Every graph $G$ with minimum degree $3$ and average degree $<4$ has either 
 \begin{itemize}
 \item two adjacent vertices with sum of degrees at most $7$,
 \item or a degree $5$ vertex with at least $4$ neighbors of degree $3.$
 \end{itemize}
\end{lem}

\begin{proof}
 Assume to the contrary that there is a counterexample.
 Assign initial charge $\deg(v)$ to every vertex $v.$
 Now we apply discharging: all vertices $v$ of degree at least $5$ give charge $\frac 13$ to each neighbor of degree $3$.
 If $\deg(v)\ge 6,$ then its new charge is at least $\frac 23 \deg(v) \ge 4.$
 If $\deg(v)=5$, the new charge is at least $\deg(v)- 3 \cdot\frac 13=4.$
 A degree $4$ vertex just keeps its initial charge of $4.$
 A degree-$3$ vertex has no neighbors with degree less than $5$ and thus gets $3$ times $\frac 13,$ so its new charge equals $3+3 \cdot \frac 13=4.$
% We conclude that every vertex ends with a charge of at least $4$, which is a contradiction with the average degree in a $K_3$-free planar graph by Euler's formula for planar graphs.
 In summary, every vertex ends with a charge of at least $4$ while the sum of the charges is preserved. This implies that the average degree of $G$ is at least $4$, contradiction.
 %However, by Euler's formula every $K_3$-free planar graph has average degree strictly smaller than $4$, contradiction.
\end{proof}

 \begin{lem}\label{lem:Nofderangements}
 \begin{enumerate}
 \item\label{itm:der_1} There are $44$ derangements of $12345$. %$\{1,2,3,4,5\}.$
 \item\label{itm:der_2} Any two permutations of $[5]$ have at least $12$ common derangements.
 \item\label{itm:der_3} Any three permutations of $[5]$ have either no common derangement, or at least two. Moreover,
 \begin{enumerate}
 \item\label{itm:der_3_a} If there are exactly two common derangements, then these are derangements of each other.
 \item\label{itm:der_3_b} If there is no common derangement, then the three permutations arise from one of them by cyclically permuting three elements and permuting the other two elements arbitrarily. 
 \end{enumerate}
 \end{enumerate} 
 \end{lem}
 \begin{proof}
 The first result is well-known.\footnote{See e.g. \url{https://oeis.org/wiki/Number_of_derangements}.}
 The other results are verified by a simple brute-force computer check.\footnote{See ~\cite[\text{K3freePlanar}]{CD23}.}
 Result~(\ref{itm:der_3_b}) is also a corollary of~\cite[Lem.~17]{CCDK23}.
 \end{proof}

The following immediate consequence of Lemma~\ref{lem:Nofderangements}(\ref{itm:der_3_b}) is of use to us. 

\begin{cor}\label{cor:NoCommonDerangement}
 For any two permutations $\sigma_1, \sigma_2$ of $[5]$, there are at most two choices for a permutation $\sigma_3$ of $[5]$ such that $\sigma_1, \sigma_2, \sigma_3$ do not have a common derangement. Moreover, if there are precisely two such choices for $\sigma_3$, then these two choices for $\sigma_3$ are not a derangement of each other. %they have 3 elements in common..
\end{cor}

Now we apply \cref{lem:Nofderangements} and~\cref{cor:NoCommonDerangement} for the number of extensions of a partial packing in a vertex, which are compatible with the prepacking in its neighbors. In the following proof, we use that the lower bound on the number of common derangements of $x$ permutations can also be used as a lower bound for the number of extensions compatible with $x$ prepacked neighbors.

\begin{thr}\label{thm:MAD<4ImpliesCorrPack<=5}
 If $\mad(G)<4$ then $\chi_c^\star(G)\le 5.$
 %Moreover, there exists a graph with $\mad(G)=4$ and $\chi_c^{\star}(G)=6$.
\end{thr}
\begin{proof}
 Assume this is not the case and let $G$ be a vertex-minimum counterexample. Consider a $5$-correspondence-assignment $L$ of $G$ for which no proper $L$-packing exists. 
 By~\cref{lem:degeneracy} the minimum degree of $G$ is $3$, and therefore the conclusion of~\cref{lem:subconf_K3freeplanar} holds for $G$.
 %\cite[Thr.~9]{CCDK21}.
 
 If $G$ has a degree-$3$ vertex $u$ which is a neighbor of a degree $\le 4$ vertex $v$,
 we can find a correspondence packing $\vec c$ of $G \setminus u$ and $L$ restricted to $V \setminus u$ by the assumption that $G$ is a minimum counterexample.
 Now consider $G$ with the partial packing $\vec c$ restricted to 
 $V \setminus \{u,v\}.$
 Since $\deg_{G \setminus u}(v)\le 3$ and there is a valid choice for $\vec c(v),$ there are at least two such choices (permutations compatible with the partial packing in the neighbors different from $u$) by Lemma~\ref{lem:Nofderangements}(\ref{itm:der_3}). 
 By Corollary~\ref{cor:NoCommonDerangement}, there are at most two valid choices $\pi_1, \pi_2$ for $\vec c(v)$ that would make it impossible to extend properly to $\vec c(u)$, and $\pi_1$ and $\pi_2$ cannot be derangements of each other. 
 If there are at least three valid choices for $\vec{c}(v)$, at least one of them makes an extension in $u$, a compatible choice of $\vec c(u)$, possible.
 If $\pi_1$ and $\pi_2$ are the only valid choices for $\vec{c}(v)$, then we obtain a contradiction because by Lemma~\ref{lem:Nofderangements}(\ref{itm:der_3_a}), $\pi_1$ and $\pi_2$ must be derangements of each other. So there must be a valid choice for $\vec c(v)$ distinct from $\pi_1$ and $\pi_2$, which we can use to extend $\vec{c}$ to $v$ and then to $u$, thus forming a packing $\vec c$ of $G$.

 Now we focus on the other case, where $G$ has a degree $5$ vertex $v$, with four neighbors $(u_i)_{i \in [4]}$ of degree $3$ and a last neighbor $u_5$.
 We will show that one can extend a proper packing $\vec c$ of $G \setminus \{v,u_1,u_2,u_3,u_4\}$ to $G$.

 For each of the $u_i$, $i \in [4]$ for which all neighbors except $v$ are prepacked, there are at most $2$ choices (by Corollary~\ref{cor:NoCommonDerangement}) of $\vec{c}(v)$ which would make it impossible to subsequently extend the packing with a valid choice of $\vec c(u_i).$
 By Lemma~\ref{lem:Nofderangements}(\ref{itm:der_1}), there are at least $44$ choices for $\vec{c}(v)$ that are not in conflict with $\vec{c}(u_5)$. So we can make one of at least $44-4\cdot 2=36$ choices for $\vec{c}(v)$ which won't cause any past or future conflicts.
 The remaining $u_i$ can now be packed in order.
 If all three neighbors of $u_i$ have currently been packed, we can extend to $\vec{c}(u_i)$ by the previous choices.
 If some (at most two) of the neighbors of $u_i$ are not packed yet, there are currently at least $12$ valid choices for $\vec{c}(u_i)$ by~\cref{lem:Nofderangements}(\ref{itm:der_2}), while by~\cref{cor:NoCommonDerangement} we need to avoid at most $2\cdot 2$ choices for $\vec{c}(u_i)$ to ensure that its unpacked neighbors can still be packed in the future. This leaves at least $8$ valid choices so we can extend in $\vec{c}(u_i)$.
 At the end, we have a proper packing $\vec c$ of $G.$
 
 Since a proper $L$-packing was derived for each of the two situations of~\cref{lem:subconf_K3freeplanar}, we always get a contradiction, i.e. no minimum counterexample exists.
\end{proof}

The condition $\mad(G)<4$ 
is sharp, because $\mad(K_5)=4$ and $\chi_c^{\star}(K_5)=6$.

\begin{cor}\label{cor:K3freePlanarCorrPack<=5}
 Let $G$ be a triangle-free planar graph. Then $\chi_c^\star(G)\le 5.$
\end{cor}
\begin{proof}
The class of planar triangle-free graphs is closed under taking subgraphs. By Euler's formula, every triangle-free planar graph has average degree strictly smaller than $4$, so the conclusion follows from Theorem~\ref{thm:MAD<4ImpliesCorrPack<=5}.
\end{proof}

\section{Planar graphs}\label{sec:planargraphs}
In this section we show that $\chi_c^{\star}(G)\leq 8$ for all graphs $G$ with $\mad(G)<6$, in particular for all planar graphs $G$.\footnote{Bonamy, Dross and La (unpublished) proved before that for a planar graph $G$, $\chi_c^\star(G)\le 9$.} 

We first prove three technical lemmas. The first of these (\cref{{lem:Nofderangements_8}}) is a variant of~\cref{lem:Nofderangements} for derangements on $[8]$ by means of counting the number of perfect matchings in some auxiliary graphs. The more perfect matchings there are, the more freedom we have in our packing arguments. In most cases~\cref{{lem:Nofderangements_8}} yields sufficiently many of them, but to analyse the exceptions, we will need  Lemmas~\ref{lem:NofForbiddenChoices8} and~\ref{lem:24perfectmatchings}.
We note that the bounds obtained by the lower matching conjecture (proven by Csikv\'{a}ri~\cite{Csikvari17}) were not sufficient for our purposes, as e.g. it gives a lower bound of only $66$ perfect matchings in a $4$-regular bipartite graph of order $16$, and after deleting multiple perfect matchings, the resulting graph is not necessarily regular.
For those reasons, the exact numbers have been computed, see~\cite[Count1factors]{CD23}.

\begin{lem}\label{lem:Nofderangements_8}
 \begin{enumerate}
 \item\label{itm:der_1} There are $14833$ derangements of $[8].$
 \item\label{itm:der_2} Any two permutations of $[8]$ have at least $4738$ common derangements.
 \item\label{itm:der_3} Any three resp. four permutations of $[8]$ have at least $1249$ resp. $248$ common derangments.
 \item\label{itm:der_5} Any five permutations of $[8]$ have either zero, or at least $33$ common derangments.
 \end{enumerate} 
 \end{lem}

 \begin{proof}
 \begin{enumerate}
 \item This is again computed before and known in general\footnote{See e.g. \url{https://oeis.org/wiki/Number_of_derangements}.}
 \item A permutation of $[8]$ corresponds with a perfect matching in a bipartite graph between two vertex sets of order $8.$
 As such, a common derangement corresponds with a perfect matching in $K_{8,8}$ after deleting two perfect matchings.
 As such, it is sufficient to consider the number of perfect matchings in edge-minimal subgraphs of $K_{8,8}$ with minimum degree $6.$ 
 These are the graphs whose bipartite complement is a union of even cycles and at most one edge.
 Up to isomorphism, there are $11$ such graphs and all of them have at least $4738$ perfect matchings. Equality is attained by the bipartite complement of a $16$-cycle (or $C_{10}\cup C_6$).
 %two suitably chosen perfect matchings.
% Equality is attained by a $6$-regular bipartite graph on $16$ vertices
% Since equality is attained by a $6$-regular bipartite graph on $16$ vertices, equality is also attained (since its bipartite complement is the union of two perfect matchings due to Hall's matchings theorem e.g.).
 \item Again we consider the number of perfect matchings in edge-minimal subgraphs of $K_{8,8}$ with minimum degree $5$ resp. $4$ and verify that all of them contain at least $1249$ resp. $248$ perfect matchings.
 Equality is attained by a (up to isomorphism unique) $5$- or $4$-regular graph.
 %Thus equality is also attained.
 \item When considering edge-minimal subgraphs of $K_{8,8}$ with minimum degree $3$, it turns out that all of them have at least $33$ perfect matchings, except that $8$ of them have none. For these, once one adds some edges such that there is at least one perfect matching, then there are at least $36.$
 Equality is attained by exactly 3 non-isomorphic $3$-regular bipartite graphs on $16$ vertices.
 Details can be found in~\cite[\text{NrOf1Factors\_mindeg3}]{CD23}. \qedhere
 \end{enumerate} 
 \end{proof}

 \begin{lem}\label{lem:NofForbiddenChoices8}
 Given $4$ permutations of $[8],$ there are at most $96$ bad permutations, i.e., choices for a fifth permutation
 such that no common derangement of the five permutations exists. 
 Moreover, if there are more than $24$ bad permutations, then there exists a permutation $\sigma$ of $[8]$ and a set $I\subseteq [8]$ of size $5$ such that every bad permutation $\tau$ satisfies $\tau(i)=\sigma(i)$ for at least four elements $i\in I$.
 \end{lem}

\begin{proof}
 Consider the matrix $A$ whose rows are the four given permutations.
 For a permutation $\sigma$, denote by $A_{\sigma}$ the matrix obtained from $A$ by adding $\sigma$ as the fifth row, and let $\Gamma_{\sigma}$ be the bipartite graph with both partite sets identified with $[8]$ (but one stands for the set of indices, and the other stands for the set of values), in which $(i,j)$ is an edge if $j$ does not appear in the $i$th column of $A_{\sigma}$. 
Observe that there exists a common derangement of the five rows of $A_{\sigma}$ if and only if $\Gamma_{\sigma}$ has a perfect matching. If $\Gamma_{\sigma}$ has no perfect matching, then we call $\sigma$ a {\em bad} permutation (with respect to $A$), and call $A_{\sigma}$ an {\em obstruction} of $A$. 
If $ \sigma$ is a bad permutation, then by Hall's marriage theorem, there exists sets of $I, J \subset [8]$ such that $N_{\Gamma_{\sigma}}(I)=J$ and $\abs{J} < \abs{I}$ (we use $I$ for a subset of indices and $J$ for a subset of values). We call $\sigma$ an $(\abs I, \abs J)$-bad permutation, and $A_{\sigma}$ an $(\abs I, \abs J)$-obstruction of $A$. 
 By a small case analysis, as done in~\cite[Lem.~18]{CCDK23}, there are only $3$ possible values for $(\abs I, \abs J)$: $(5,3),(5,4)$ and $(4,3).$ 
 We now consider three cases, and when we consider a case, it is assumed that earlier cases do not apply.

 \textbf{Case 1: $A$ has a $(5,3)$-obstruction.}\\
 Let $\sigma$ be a $(5,3)$-bad permutation with respect to $A$. Let $I, J \subseteq [8]$ be sets such that $|I|=5$, $|J|=3$ and $N_{\Gamma_{\sigma}}(I)=J$. Then each of the five values in $\bar{J}$ (where $\bar{J} := [8] \setminus J$) appears in every column of $A_{\sigma}$ indexed by $I$. 
Thus the submatrix $A_{\sigma}[I]$ of $A_{\sigma}$ consisting of columns indexed by $I$ is a $5 \times 5$ Latin square, and $A[I]$ is a $4 \times 5$ Latin rectangle. It is easy to see that $A$ contains no other $4 \times 5$ Latin rectangle. Hence 
for any permutation $\tau$ (of $[8]$), $A_{\tau}$ is a $(5,3)$-obstruction of $A$ if and only if %$\tau[\bar{J}] = \sigma[\bar{J}]$, which means that each element of $[\bar{J}]$ has the same position in $\tau$ and $\sigma$. 
$\tau(i)=\sigma(i)$ for all $i\in [8]$ such that $\sigma(i) \in \bar{J}$.
There are $3!=6$ such $(5,3)$-bad permutations. For each of these $6$ permutations, for $s \in \bar{J}$ and $t \in J$, 
interchanging $s$ and $t$, we obtain a $(5,4)$-bad permutation. 
This gives $5 \times 6 \times 3$ bad permutations. For any other permutation $\theta$ of $[8]$, $A_{\theta}[I]$ extends $A[I]$ to a $5\times 5$ matrix that has at least two entries from $J$. By Hall's matchings theorem (or a small case analysis), $\Gamma_{\theta}$ has a perfect matching, and hence $\theta$ is not a bad permutation. 
This can also be verified with checking a worst-case scenario up to isomorphism, $A=\begin{bmatrix}
 5 & 1 & 2 & 3 & 4 & 7 & 8 & 6 \\
4 & 5 & 1 & 2 & 3 & 6 & 7 & 8 \\
 3 & 4 & 5 & 1 & 2 & 8 & 6 & 7 \\
2 & 3 & 4 & 5 & 1 & 6 & 8 & 7 \\
\end{bmatrix}.$\\
 Thus in total there are $5\times 6 \times 3 +6=96$ bad permutations with respect to $A$.

 \textbf{Case 2: $A$ has a $(5,4)$-obstruction.}\\ 
Let $\gamma$ be a $(5,4)$-bad permutation with respect to $A$. Let $I, J \subseteq [8]$ be subsets such that $|I|=5$, $|J|=4$ and $N_{\Gamma_{\gamma}}(I)=J$. 
 Then each column of $A_{\gamma}[I]$ contains exactly one element of $J$ (not all the columns contain the same element of $J$ since we are not in Case 1). 
 For a matrix $M$ and a set $X$ of values, we denote by $M(X,\star)$ the matrix obtained from $M$ by replacing each element not in $X$ by a $\star$. Then $A_{\gamma}[I](\bar{J},\star)$ is a Latin square with elements $\bar{J} \cup \{\star\}$, and $A[I](\bar{J}, \star)$ is a $4 \times 5$ Latin rectangle. It is easy to check that for any other index set $I'$ with $|I'|=5$ and any subset $X$ of $[8]$ with $|X|=4$, $A[I'](X, \star)$ is not a Latin rectangle. Hence for any permutation $\tau$ of $[8]$, we have that $\tau$ is a $(5,4)$-bad permutation with respect to $A$ if and only if $\tau(i)=\gamma(i)$ for all $i\in[8]$ such that $\gamma(i)\in \bar{J}$.
 %$\tau[\bar{J}] = \gamma[\bar{J}]$.
 So there are $4!$ $(5,4)$-bad permutations with respect to $A$. 
 If there exists no $(4,3)$-bad permutation with respect to $A$, then we are done and we take $\sigma=\gamma$. 

Otherwise, if there is also a $(4,3)$-bad permutation $\theta$, let $I', J' \subseteq [8]$ be subsets such that $|I'|=4$, $|J'|=3$ and $N_{\Gamma_{\theta}}(I')=J'$. Then $A_{\theta}[I']$ is a $5 \times 4$ Latin rectangle with element set $\bar{J'}$. It is easy to check that this implies that $I' \subseteq I$ and $J' \subseteq J$. Moreover, $4$ copies of the symbol $\star$ in $A_{\gamma}[I](\bar{J}, \star)$ stand for the (single) element $j^{\star}$ in $J \setminus J'$ and one copy of $\star$ stands for an element in $J'$. 
 Thus the subsets $I'$ and $J'$ are unique,
 and $A[I']$ is the unique submatrix of $A$ that is a $4 \times 4$ that can be extended to a $5 \times 4$ Latin rectangle. 
 
We now take $\sigma$ such that $A_{\sigma}[I']$ is a $5 \times 4$ Latin rectangle and $\sigma( I \setminus I')$ equal to the unique element in $\bar{J'}$ that does not occur in $\sigma[I']$. 
 That is, we take a permutation $\sigma$ that is simultaneously a $(I,J)$-bad permutation and a $(I',J')$-bad permutation.
Since $\sigma$ is $(I,J)$-bad, it holds (similar to what we had derived for $\gamma$) that every $(5,4)$-bad permutation $\tau$ satisfies $\tau(i)=\sigma(i)$ for all $i\in [8]$ such that $\sigma(i) \in \Bar{J}$, so in particular for at least four elements $i\in I$. On the other hand, since $\sigma$ is $(I',J')$-bad we have the following.
If $\tau$ is a permutation of $[8]$, then $\tau$ is a $(4,3)$-bad permutation with respect to $A$ only if $\tau(i)=\sigma(i)$ for $ i \in I'$. Moreover, among the $4!$ permutations $\tau$ for which $\tau(i)=\sigma(i)$ for $i \in I'$, $3!$ permutations are counted in the $4!$ $(5,4)$-bad permutations. 
 So altogether, there are at most $4!+ (4!-3!) =42$ bad permutations with respect to $A$.

An example of the latter situation is given by the matrix
 $A=\begin{bmatrix}
 5 & 1 & 2 & 3 & 4 & 7 & 8 & 6 \\
4 & 5 & 1 & 2 & 3 & 6 & 7 & 8 \\
 3 & 4 & 5 & 1 & 2 & 8 & 6 & 7 \\
2 & 3 & 4 & 6 & 1 & 5 & 8 & 7 \\
\end{bmatrix},$
where e.g. $\sigma=\begin{bmatrix} 1 & 2 & 3 & 4 & 5 & 6 & 7 & 8 \end{bmatrix}$ is both $(\{1,2,3,4,5\},\{5,6,7,8\})$-bad and $(\{1,2,3,5\},\{6,7,8\})$-bad with respect to $A$.

\textbf{Case 3: $A$ has a $(4,3)$-obstruction.} \\
 Let $\sigma$ be a $(4,3)$-bad permutation with respect to $A$. Let $I, J \subseteq [8]$ be subsets such that $|I|=4$, $|J|=3$ and $N_{\Gamma_{\sigma}}(I)=J$. 
 Then $A_{\sigma}[I]$ is a $5\times 4$ Latin rectangle. 
 (Note that $A$ cannot have multiple $(4,3)-$obstructions.) 
 Similarly as in the previous two cases, a permutation $\tau$ is a $(4,3)$-bad permutation with respect to $A$ if and only if $\tau_i=\sigma_i$ for $i \in I$. Therefore there are $24$ $(4,3)$-bad permutations with respect to $A$.\\
Finally, note that if there are more than $24$ bad permutations, we are in case $1$ or $2$, and by our choice of the permutation $\sigma$ of $[8]$ there exists a set $I\subseteq [8]$ of size $5$ such that every bad permutation $\tau$ satisfies $\tau(i)=\sigma(i)$ for at least four elements $i\in I$.
\end{proof}

\begin{lem}~\label{lem:24perfectmatchings}
Let $\Gamma=(A\cup B, E)$ be a bipartite graph with minimum degree $\geq 3$ and both parts of size $8$. If $\Gamma$ has more than $24$ perfect matchings then for every five pairs in $A\times B$ there is some perfect matching that contains fewer than four of these pairs.
\end{lem}
\begin{proof}
With some abuse of notation we identify both $A$ and $B$ with the set $[8]$, and $(i,j)\in E$ refers to an edge between $i\in A$ and $j\in B$.
Suppose to the contrary that there exist five pairs in $A\times B$ (without loss of generality $(i,i)$, $1\leq i \leq 5$) such that every perfect matching contains at least four of them. Observe that $(i,i) \in E$ for every $i \in [5],$ since the $4$ elements matched with itself cannot be the same $4$ every time; otherwise $\Gamma$ would have at most $(8-4)!=24$ perfect matchings, contradiction.

\begin{claim}\label{claim:matchingclaim1}
 Without loss of generality $(i,i)\in E$ for $i \in [7]$, and $(1,8),(8,1) \in E$
\end{claim}
\begin{claimproof}
Consider a perfect matching $M$ which intersects precisely four out of five edges $(i,i), i \in [5]$. Up to relabeling $[5]$ in both $A$ and $B$, we have $(i,i) \in M$ for all $i \in \{2,3,4,5\}$ but not $(1,1)\in M$. Up to relabeling $\{6,7,8\}$ in $B$ (but not in $A$), we may then assume that $(6,6),(7,7),(1,8),(8,1) \in M$. Since also $(1,1)\in E$ by assumption, we conclude.
\end{claimproof}

We now separate the analysis based on the condition of Hall being strictly met or not. This condition (which holds because $\Gamma$ has a perfect matching) says that for every $I \subset A=[8],$ $N(I) \subset B$ satisfies $\abs{N(I)} \ge \abs {I}.$ \\

If there is a non-empty strict subset $I \subset A= [8]$ with $\abs{I}=\abs{N(I)},$ 
then $I \cup N(I)$ spans a bipartite graph in which each vertex in $I$ has degree at least $3$, and $(A \setminus I) \cup (B \setminus N(I))$ spans a bipartite graph in which each vertex in $B\setminus N(I)$ has degree at least 3. In particular, $\abs{I} \in \{3,4,5\}$. By~\cref{thr:MarshallHall}, in both bipartite graphs there are at least $6$ perfect matchings, and therefore at least one perfect matching in $\Gamma$ contains at most $3$ edges of the form $(i,i), i \in [5],$ contradiction.
So from now on we may assume that $\abs{N(I)}>\abs{I}$ for every non-empty strict subset $I \subset A$. 

\begin{claim}\label{claim:matchingclaim2}
 There is no edge $(i,j) \in E$ with $1 \le i <j \le 5.$
\end{claim}
\begin{claimproof}
Suppose there is such an edge $(i,j)$. Since the induced subgraph obtained from $\Gamma$ by removing the endvertices of $(i,j)$ satisfies Hall's condition, it has a perfect matching.
Hence $\Gamma$ has a perfect matching containing $(i,j)$. This matching avoids $(i,i)$ and $(j,j)$, contradiction. 
\end{claimproof}

\begin{claim}\label{claim:matchingclaim3}
 There are no values $2 \le i \le 5 <j \le 7$ for which both $(i,j),(j,i) \in E$. Hence $(i,8),(8,i) \in E$ for every $i \in [5]$.
\end{claim}
\begin{claimproof} 
 If there is a pair $(i,j),(j,i) \in E$ where $2 \le i \le 5 <j \le 7$ (e.g. the dashed lines in~\cref{fig:Some_configurations_caseanalysis_mad6}),
 then the perfect matching $\{(1,8),(8,1),(i,j),(j,i)\} \cup \{(k,k) \mid k \in [7] \setminus\{1,i,j\}\}$ does not use $(1,1)$ nor $(i,i)$, contradiction.
 Let $i\in [5]$. Suppose for a contradiction that $(i,8) \notin E$. Since $\Gamma$ has minimum degree $\geq 3$, Claim~\ref{claim:matchingclaim2} implies that $(i,6),(i,7)$ are edges, and that at least one of $(6,i), (7,i)$ is an edge, contradiction. So $(i,8)$ is an edge and by the same argument $(8,i)$ is an edge.
\end{claimproof}

\begin{claim}\label{claim:matchingclaim4}
 Without loss of generality, the edges incident to $(i,i)$ are $(i,i),(i,8),(8,i),(i,6),(7,i)$, for every $i \in [5].$
\end{claim}
\begin{claimproof}
By the previous claims, we only need to investigate which of the pairs $(i,6),(i,7),(6,i),(7,i)$ are present. 
 By Claim~\ref{claim:matchingclaim3} and minimum degree $3$, for every $i\in [5]$, precisely one of $(i,6),(i,7)$ is an edge, and precisely one of $(6,i),(7,i)$ is an edge.
If there exist distinct $i,j \in [5]$ such that all of $(i,6),(7,i),(6,j),(j,7)$ are edges (see the dotted lines in~\cref{fig:Some_configurations_caseanalysis_mad6} for an example), then
for every choice of $h \in [5]\setminus \{i,j\}$, the perfect matching $\{(h,8),(8,h),(i,6),(7,i),(6,j),(j,7) \} \cup \{(k,k) \mid k \in [7] \setminus\{h,i,j,6,7\}\}$ does not use $(h,h)$ nor $(i,i)$ (nor $(j,j)$), contradiction. 
Therefore either $(i,6),(7,i)\in E$ for all $i\in [5]$, or $(i,7),(6,i)\in E$ for all $i\in [5]$. Without loss of generality (if necessary relabeling $\{6,7\}$ in both parts of the bipartition), we may assume the former.
\end{claimproof}

By~\cref{claim:matchingclaim4} there are no edges of the form $(6,i)$, $i\in [5]$, so minimum degree $\geq 3$ implies $(6,7)\in E$. Combined with the other edges guaranteed by~\cref{claim:matchingclaim4} this reveals the perfect matching $\{(5,6),(6,7),(7,5),(1,8),(8,1) \cup \{(k,k) \mid 2 \le k \le 4\}$ avoiding both $(1,1)$ and $(5,5)$, contradiction. 
\qedhere
\end{proof}

 \begin{figure}[h]
 \centering
 \begin{tikzpicture}
 \foreach \i in {1,2,...,7}{
 \draw[fill] (\i,2) circle (0.1);
 \draw[fill] (\i,0) circle (0.1);
 \node at (\i,-0.3) {${\i}$};
 \node at (\i,2.3) {${\i}$};
 }

 \node at (0,-0.3) {${8}$};
 \node at (0,2.3) {${8}$};
 \draw[fill] (0,2) circle (0.1);
 \draw[fill] (0,0) circle (0.1);

 \foreach \i in {1,2,...,7}{
 \draw[thick] (\i,2)--(\i,0);
 }

 \draw[dashed] (3,0)--(6,2);
 \draw[dashed] (3,2)--(6,0);
 
 \Edge (1,0)(0,2); %[style={ bend right}]
 \Edge (0,0)(1,2);

 \Edge[style=dotted](6,2)(5,0);
 \Edge[style=dotted](5,2)(7,0);
 \Edge[style=dotted](7,2)(4,0);
 \Edge[style=dotted](4,2)(6,0);
 
 \end{tikzpicture}
 \caption{Some configurations of possible edges present in $\Gamma$}
 \label{fig:Some_configurations_caseanalysis_mad6}
 \end{figure}
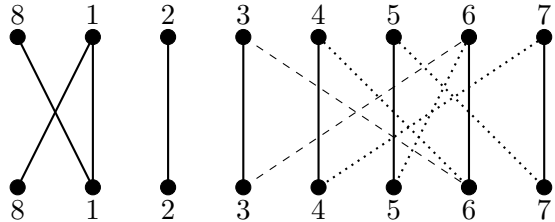

Having laid the technical groundwork in Lemmas~\ref{lem:Nofderangements_8},~\ref{lem:NofForbiddenChoices8} and~\ref{lem:24perfectmatchings}, we are now prepared to assemble the main theorem of this section.

\begin{thr}\label{thm:MAD<6ImpliesCorrPack<=8}
 If $\mad(G)<6$, then $\chi_c^\star(G)\le 8.$
\end{thr}

\begin{proof}
 We first need a structural result.
 \begin{claim}\label{clm:subconf_mad6}
 Every graph with minimum degree $5$ and average degree $<6$ has either 
 \begin{itemize}
 \item two adjacent vertices with sum of degrees at most $11$,
 \item or a degree-$7$ vertex with at least $6$ neighbors of degree $5.$
 \end{itemize}
 \end{claim}

 \begin{claimproof}
 This is analogous to~\cref{lem:subconf_K3freeplanar},
 where the initial charge of every vertex is $\omega(v)=\deg(v)$ and every vertex of degree $\geq 7$ gives charge $1/5$ to each neighbor of degree $5$.
 % Initial charge $\omega(v)=\deg(v)$. Then every vertex of degree $\geq 7$ gives charge $1/5$ to each neighbor of degree $5$. In the end every vertex has charge $\geq 6$, so average degree $\geq 6$, contradiction.
 \end{claimproof}

 Now, let $G$ be a minimum counterexample to the theorem.
 Consider a $8$-correspondence-cover $(L,H)$ of $G$ which does not admit a packing.
 By~\cref{lem:degeneracy}, the minimum degree of $G$ is at least $5$.
 By~\cref{clm:subconf_mad6}, we only need to consider the following three cases.

 %\begin{enumerate}
 %\item 
 
 \textbf{1. $G$ has a vertex $v$ of degree $7$, with at least $6$ neighbors $u_i, i \in [6]$ of degree $5$.}
 
 Then consider a partial packing $\vec{c}$ on $G \backslash \{v,u_1,u_2, \ldots, u_6\}.$
 Since there are $14833$ derangements of $[7]$, there are also this many choices for $\vec c(v)$ which are compatible with $\vec c(u_7)$, where $u_7$ is the seventh neighbor of $v$.
 On the other hand, for any $i \in [6],$ by~\cref{lem:NofForbiddenChoices8} there are no more than $96$ choices for $\vec c(v)$ for which the partial packing cannot be extended in $u_i$ 
 (note that there are no such bad choices if at most $3$ neighbors of $u_i$ are packed already). Since $14833-6\cdot 96 >0$, we can make a \emph{safe} choice for $\vec{c}(v)$ that allows a subsequent extension to any single vertex $u_i$, $1\leq i \leq 6$. 

 Next, for $1 \le i \le 6$ in order, we will make a \emph{safe} choice for $\vec{c}(u_i)$ such that subsequently for all $j >i$, there exists at least one further extension of the packing to $u_j$. Note that at this point in the process we will merely impose the existence of extensions to one further vertex $u_j$, not yet the existence of an extension to simultaneously all of $u_{i+1}, \ldots, u_{6}$. 

 If $u_i$ has precisely $4$ (resp. $3$, resp. $2$, resp. $1$) neighbors that are not yet packed, then by~\cref{lem:Nofderangements_8} and~\cref{lem:NofForbiddenChoices8}, the number of safe extensions in $u_i$ is at least $14833-4\cdot 96$ (resp. $4738-3\cdot 96$, resp. $1249-2\cdot 96$, resp. $248-96$), which is positive in all cases.
If all neighbors of $u_i$ are already packed, then by our safe choices of the partial packing on the neighbors in $\{v,u_1,\ldots,u_{i-1}\}$, there is at least one (safe) extension to $u_i$, and in fact by~\cref{lem:Nofderangements_8} the number of extensions is at least $33$.

 \textbf{2. $G$ has two adjacent vertices $u,v$ with degree $5$.}
 
 Let $\vec c$ be a partial packing of $G \setminus \{u,v\}.$
 Given the partial packing $\vec c$ of four neighbors of $u$, there are at least $248$ ways to extend it to $u$ by~\cref{lem:Nofderangements_8} (\ref{itm:der_3}).
 On the other hand, there are at most $96$ choices for $\vec c(u)$ after which there is no valid way to choose for $\vec c(v)$ anymore by \cref{lem:NofForbiddenChoices8}.
 As such, there are at least $248-96$ ways to choose $\vec c(u)$, after which there are at least $33$ choices for $\vec c(v)$ by~\cref{lem:Nofderangements_8}(\ref{itm:der_3}).

 %\item 
 \textbf{3. $G$ has two adjacent vertices $u,v$ with degree $5$ and $6$ respectively.}

 This is the critical case.
 Let $\vec{c}$ be a packing of $G \setminus \{u,v\}$ obtained from a packing of $G \setminus u$ by uncoloring $v$. 
 Without loss of generality, we can assume that the matchings between $v$ and its neighbors are the (full) identity matchings.
 Let $u_1,u_2,\ldots, u_5$ be the other 5 neighbors of $v$. Let $\Gamma=(A \cup B,E(\Gamma))$ be the bipartite graph with both bipartition classes $A,B$ being a copy of $[8],$ with $(i,j) \in E(\Gamma)$ if $\vec c(u_l)_i \ne j$ for all $l \in\{1,2,\ldots, 5\}$. A perfect matching in $\Gamma$ corresponds to an extension of $\vec{c}$ to $v$. As $\vec{c}$ is obtained from a packing of $G\setminus u$ by uncoloring $v$, there is an extension of $\vec{c}$ to $v$. So $\Gamma$ has at least one perfect matching.
 %, and the Hall's matching conditions are met. 
 Note that $\Gamma$ has minimum degree at least $8-5=3$.
 
 By~\cref{lem:Nofderangements_8} (\ref{itm:der_5}), there are at least $33$ choices for $\vec c(v)$ that extend $\vec{c}$ to $v$. 
On the other hand, by \cref{lem:NofForbiddenChoices8}, except in exceptional cases, we have at most $24$ \emph{bad choices} for $\vec{c}(v)$ that do not allow subsequent extension to $u$. 

It remains to consider these exceptional cases in which, by~\cref{lem:NofForbiddenChoices8}, we can assume that the (unfortunately more than $33$) bad choices for $\vec{c}(v)$ are a subset of the permutations on $[8]$ that fix at least $4$ out of $5$ elements in $[5]$. 
Luckily, since $\Gamma$ has minimum degree $\geq 3$ and has $\geq 33>24$ perfect matchings, our~\cref{lem:24perfectmatchings} implies the existence of a choice for $\vec{c}(v)$ that fixes fewer than $4$ out of those $5$ elements. This choice is not bad.
\end{proof}

\begin{rem}
 For every graph $G$, every $k$-fold correspondence-cover with $k\geq \chi_{c}^\star(G)+1$ admits exponentially many correspondence-colorings (since $k \ge \chi_c(G)+1$) and thus also exponentially many packings (since in this case each coloring can be extended to a packing, while each packing contains only $k$ colorings). Therefore our result immediately implies that every $9$-fold cover of a planar graph admits exponentially many packings. However, the proof above can actually be modified to prove that the same is true for every $8-$fold cover.
 Proof sketch: 
 The critical case is when $u,v$ have degree $5,6.$
 Given a partial packing of $G \setminus \{u,v\},$
 let $q_v$ and $q_{uv}$ be the number of valid extensions to $G \setminus \{u\}$ and $G$ respectively.
 One can prove that $q_{uv} \ge \frac{99}{98}q_u$ with a careful case analysis\footnote{The constant $\frac{99}{98}$ is not optimal.}. 
 E.g. if there is no perfect matching (p.m.) which is the identity on $[5],$ there are at most $30$ p.m.'s that leave $4$ of the $5$ elements of $[5]$ identical.
 In this case $q_{uv} \ge 33(q_u-30)\ge 3q_u$ (since $q_u \ge 33$), as there are at least $q_u-33$ valid choices for $\vec c(v)$ which can be extended in at least $33$ ways in $u.$
 %We count the number of p.m.'s that do not leave 4 of the 5 elements of $[5]$ identical.
 % If some $\abs{N(I)}=\abs{I},$ there are many choices (at least $16$). 
 % In the other case, there are at least $3$ 
\end{rem}

It was proven by Mader~\cite{Mader68} that $K_5$-minor-free graphs $G$ satisfy $\mad(G)<6$. Hence we have the following corollary.

\begin{cor}
 If $G$ is a $K_5$-minor-free graph, then $\chi_c^\star(G)\le 8.$ In particular, every planar graph $G$ has $\chi_c^\star(G)\le 8$. 
\end{cor}

\section{Technical Lemmas for fractional packing}\label{sec:TechLemmasforFracPacking}
This section contains our main tools for fractional packing. We show that in order to inductively bound the fractional packing number, it suffices to find an induced subgraph that has smaller fractional packing number and few neighbors in the remainder of the graph.
We first make a simple yet useful observation: increasing the size of a list preserves the property of having a fractional packing, even if as a result not all lists have the same size.\\

%A fractional packing of a cover $\sH=(L,H)$ of a graph $G$ is equivalent with a probability distribution on the independent transversals $I$ of $\sH$ for which $\Pr(x \in I)=\frac{1}{\abs{L(v)}}$ for every $x \in L(v).$ We refer to (an independent transversal chosen according to) such a probability distribution as a random independent transversal.

Given a probability distribution on independent transversals of a cover $\sH=(L,H)$, we refer to (an independent transversal chosen according to) such a probability distribution as a \emph{random independent transversal}.
Recall that a fractional packing of a cover $\sH=(L,H)$ of a graph $G$ is equivalent with a probability distribution on the independent transversals $I$ of $\sH$ for which $\Pr(x \in I)=\frac{1}{\abs{L(v)}}$ for every $x \in L(v).$

\begin{lem}\label{lem:monotonicityfractionalpacking}
Let $G$ be a graph and let $s: V(G) \rightarrow \mathbb{N}$ be a function. Suppose every correspondence-cover $\sH^{-}=(L^{-},H^{-})$ of $G$ with $|L^{-}(v)| = s(v)$ for all $v$ has a fractional packing. 
Then also every correspondence-cover $\sH=(L,H)$ of $G$ with $|L(v)| \geq s(v)$ for all $v$ has a fractional packing. The same holds mutatis mutandis for list-covers.
\end{lem}

\begin{proof}
We proceed by induction on $|V(H)|=\sum_{v\in V(G)}|L(v)|$. The base case where $|L(v)|=s(v)$ for every $v$ holds by assumption. So suppose $|L(v)|> s(v)$ for some vertex $v$. 
For a fixed $x\in L(v)$, consider the correspondence-assignment $L_x$ given by $L_x(v)=L(v)-x$ and $L_x(w)=L(w)$ for every other vertex $w\in V(G)-v$. By the induction hypothesis the correspondence-cover $(L_x,H-x)$ has a fractional packing, meaning in particular that it has a random independent transversal $I_x$ such that $\mathbb{P}(y \in I_x)=\frac{1}{|L_x(v)|}=\frac{1}{|L(v)|-1}$ for every $y\in L_x(v)$. Note that $I_x$ is also a random independent transversal of $\sH$.
%Now we obtain a random independent transversal $I$ of $\sH$ by first selecting a uniformly random color $x \in L(v)$, each with probability $1/|L(v)|$, and then choosing $I=I_x$.
Now we obtain a random independent transversal $I$ of $\sH$ by selecting $I$ according to the probability distribution of $I_x$ with probability $1/|L(v)|$, for each $x\in L(v)$. Crudely this can be viewed as first selecting a uniformly random color $x \in L(v)$ and then selecting $I$ according to $I_x$.
We compute that for every vertex $y\in L(v)$ the desired probability distribution emerges: we have $\mathbb{P}(y \in I) = \mathbb{P}(y \neq x) \cdot \mathbb{P}(y \in I_x) = \frac{|L(v)|-1}{|L(v)|} \cdot \frac{1}{|L(v)|-1} = \frac{1}{|L(v)|}$. Moreover, for every $w\in V(G)-v$ and $y \in L(w)$ the probability distribution was unaffected by the choice of $x$, so that we still have $\mathbb{P}(y\in I)=\mathbb{P}(y\in I_x)=1/|L(w)|$. Thus $I$ certifies that $\sH$ has a fractional packing.
\end{proof}

\begin{cor}\label{cor:monotonicityfractionalpacking}
Let $\sH=(L,H)$ be a correspondence-cover of a graph $G$ such that $|L(v)|\geq \chi_{c}^{\bullet}(G)$ for every vertex $v$ of $G$. Then $\sH$ has a fractional packing. The same holds mutatis mutandis for $\chi_{\ell}^{\bullet}(G)$ and list-covers.
\end{cor}

The next Lemma describes a broad setting in which it is possible to apply induction on a locally modified graph. Its useful Corollaries~\ref{cor:intermediatecorollary} and~\ref{cor:nosparselyconnectedsubgraph} might be easier to digest for the reader.

\begin{lem}\label{lem:technicalinduction}
Let $G$ be a graph with a correspondence-cover $\sH=(H,L)$. If $G$ contains an induced subgraph $T$ satisfying the following four conditions, then $\sH$ has a fractional packing.

\begin{enumerate}[(i)]
 \item \label{item:TechInd1} Every vertex of $T$ has at most one neighbor outside $T$;
 \item \label{item:TechInd2} $2 \leq |L(u)|\leq |L(v)|$ for every $u \in V(T)$ with one neighbor $v$ outside $T$.
 \item \label{item:TechInd3} The restriction of $\sH$ to $G\setminus T$ has a fractional packing;
 \item \label{item:TechInd4} Every correspondence-cover $(H_T,L_T)$ of $T$ with 
\begin{itemize}
 \item $|L_T(u)| = |L(u)|$, for every $u\in V(T)$ with no neighbor outside $T$, and
 \item $|L_T(u)| = |L(u)|-1$, for every $u\in V(T)$ with one neighbor $v$ outside $T$
\end{itemize}
has a fractional packing.
\end{enumerate}
\end{lem}
\begin{proof}
The fractional chromatic number of a graph does not decrease under edge addition, so by adding edges to the cover $H$ if necessary, we may assume that for every edge $uv$ of $G$, the matching between $L(u)$ and $L(v)$ is maximum, i.e.\ of size $\min \{|L(u)|, |L(v)|\}$.\\

Let $\sH_0=(H_0,L_0)$ be the cover of $G_0:=G\setminus T$ defined by $H_0=H-\bigcup_{u\in V(T)}L(u)$, and $L_0$ the list-assignment $L$ restricted to $G_0$. 
By property~\ref{item:TechInd3}, there exists a random independent transversal $I_0$ of $\sH_0$ such that for every vertex $v \in V(G_0)$ and every color $x\in L(v)$ in its list: 
 $$\mathbb{P}(x\in I_0) = 1/|L_0(v)|=1/|L(v)|.$$ 

Let $\sH_T=(H_T,L_T)$ be the (random) correspondence-cover of $T$ given by the cover-graph $H_T:=H-N[I_0]$, and with respect to the reduced lists given by $L_T(u):=L(u)-N[I_0]$, for all $u \in V(T)$. 

For each $u\in V(T)$, if $u$ has precisely one neighbor outside $T$ then with respect to $H$, $L(u)$ has at most one neighbor in $I_0$, so $|L_T(u)|\geq |L(u)|-1$. Otherwise (by~\ref{item:TechInd1}) $L(u)$ has no neighbor in $I_0$, so $|L_T(u)|=|L(u)|$. By Lemma~\ref{lem:monotonicityfractionalpacking} and property~\ref{item:TechInd4}, we conclude that $\sH_T$ has a fractional packing. This means that there exists a random independent transversal $I_T$ of $\sH_T$ such that for all $u \in V(T)$ and all $x\in L_T(u)$: 
 $$\mathbb{P}(x\in I_T) = 1/|L_T(u)|.$$
 
We now choose the union $I:= I_0 \cup I_T$ as our random independent transversal of $\sH$. We need to show that $I$ certifies that $\sH$ admits a fractional packing.\\

For vertices outside $T$ it is straightforward: for every $v \in V(G_0)$ and every $x\in L(v)$, we simply have $\mathbb{P}(x\in I)=\mathbb{P}(x\in I_0)=\frac{1}{|L(v)|}$, as desired.
For any vertex $u\in V(T)$ without a neighbor outside $V(T)$ we have $L_T(u)=L(u)$ regardless the random choice of $I_0$, so equally simply we get (for every color $x\in L(u)$) that
$\mathbb{P}(x\in I)=\mathbb{P}(x\in I_T)=\frac{1}{|L_T(u)|}=\frac{1}{|L(u)|}$, as desired.\\

We are left with the situation that $u \in V(T)$ has precisely one neighbor $v$ outside $T$. Note that with respect to $H$, at most one color in $L(u)$ has a (unique) neighbor in the random set $I_0$. Let $x \in L(u)$ be fixed. Depending on the nature of $I_0$, we distinguish three conditional probabilities:

\begin{itemize}
 \item Conditional on $I_0 \cap N(x) \neq \emptyset$, we have $x \notin L_T(u)$, so the probability that $x \in I_T$ is zero;
 \item Conditional on $I_0 \cap N(L(u)-x) \neq \emptyset$, the probability that $x \in I_T$ is $1/|L_T(u)|=1/(|L(u)|-1)$;
 \item Conditional on $I_0 \cap N(L(u))= \emptyset$, the probability that $x \in I_T$ is $1/|L_T(u)|=1/|L(u)|$.
\end{itemize}

%The probability of the first case is irrelevant for our final calculation, as the conditional probability is zero.
By property~\ref{item:TechInd2} and the assumption that the matching between $L(u)$ and $L(v)$ is maximum, we have $| L(v) \cap N(S) |=|S|$, for every $S\subseteq L(u)$. 
Therefore the probability of the second case is $$\Pr(I_0 \cap N(L(u)-x) \neq \emptyset)= \frac{|N(L(u)-x)|}{ |L(v)|}= \frac{|L(u)|-1}{ |L(v)|},$$
and the probability of the third case is
$$\Pr(I_0 \cap N(L(u)) = \emptyset)= \frac{|L(v) - N(L(u))|}{|L(v)|} = \frac{|L(v)|- |L(u)|}{|L(v)|}.$$

Combining the conditional probabilities, we obtain
$$\Pr(x \in I) =\frac{1}{|L(u)|-1} \cdot \frac{|L(u)|-1}{|L(v)|} + \frac{1}{|L(u)|} \cdot \frac{|L(v)|-|L(u)|}{|L(v)|} = \frac{1}{|L(u)|},$$
taking heed of the fact that we did not divide by zero in this calculation, due to property~\ref{item:TechInd2}. This concludes the proof that $\sH$ has a fractional packing.
\end{proof}

\begin{cor}\label{cor:intermediatecorollary}
Let $G$ be a graph with a $k$-fold correspondence-cover $\sH=(H,L)$, for some $k\geq 2$.
Suppose $G$ contains an induced subgraph $T$ such that every vertex of $T$ has at most one neighbor outside $T$, and such that the restriction of $\sH$ to $G\setminus T$ has a fractional packing. Suppose furthermore that every correspondence-cover $(H_T,L_T)$ of $T$ has a fractional packing whenever for all $v\in V(T)$: 
\[ |L_T(v)| = 
\begin{cases} 
k
 & \text{if } v \text{ has no neighbor outside } T \\
k-1
 & \text{if } v \text{ has one neighbor outside } T. \\
\end{cases}\]
Then $\sH$ also has a fractional packing.
\end{cor}

\begin{cor}\label{cor:nosparselyconnectedsubgraph}
 Let $r\geq 2$ be an integer and let $\G$ be a graph class.
 Then a vertex-minimum graph $G \in \G$ with $\chi_c^\bullet(G)>r$ does not contain any induced subgraph $T$ such that
 \begin{itemize}
 \item $G\setminus T$ belongs to $\G$, and
 \item $\chi_c^\bullet(T)\leq r-1$, and 
 \item every vertex of $T$ has at most one neighbor outside $T$. 
 \end{itemize}
\end{cor}
\begin{proof}
 Let $G=(V,E) \in \G$ be a vertex-minimal graph for which there is an $r$-fold correspondence-cover $\sH=(L,H)$ that does not admit a fractional packing. Suppose for a contradiction that $G$ contains an induced subgraph $T$ of the described form. Because $G$ is minimal, the restriction of $\sH$ to $G\setminus T$ admits a fractional packing. Combining $\chi_c^{\bullet}(T) \leq r-1$ with Lemma~\ref{lem:monotonicityfractionalpacking}, it follows that every cover of $T$ as described in the condition of Corollary~\ref{cor:intermediatecorollary} admits a fractional packing. So by Corollary~\ref{cor:intermediatecorollary}, $\sH$ has a fractional packing; contradiction.
\end{proof}
\begin{rem}
As we will see, Corollary~\ref{cor:nosparselyconnectedsubgraph} is particularly useful for deriving $\chi_{c}^{\bullet}(G)\leq r$ for every graph $G$ in some subgraph-closed graph class $\G$. In that case it suffices to show that every $G\in \G$ contains an induced tree $T$ whose vertices each have at most one neighbor outside $T$. Let us emphasize that $T$ could be arbitrarily large and we only care about its existence, thus allowing for a non-local argument.
\end{rem}

\begin{rem}
The condition in~\cref{lem:technicalinduction} and its corollaries that no vertex of $T$ has more than one neighbor outside $T$ is necessary to a surprising extent. One might think that $t$ neighbors outside could be allowed at the expense of further increasing the list sizes by $t$.
This turns out to be impossible in general. E.g. for $t=2$ and $k=4$ in~\cref{cor:intermediatecorollary}, $T=K_2$ (which satisfies $\chi_c^\bullet(T)=2$), 
there exists a fractional packing of the neighborhood of $T$ which cannot be extended to $T$. See~\cref{fig:nonextendablefractionalpacking}.
\end{rem}
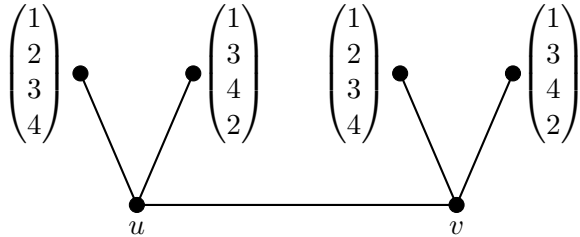
\begin{figure}[h]
 \centering
 \begin{tikzpicture}
 \foreach \x in {0,4.25}
 {
 \draw[fill] (\x,0) circle (0.1);
 \draw[fill] (\x-0.75,1.75) circle (0.1);
 \draw[fill] (\x+0.75,1.75) circle (0.1);
 \draw[thick] (\x+0.75,1.75)--(\x,0) -- (\x-0.75,1.75);
 }
 \draw[thick] (0,0)--(4.25,0);

 \node at (0,-0.3) {$u$};
 \node at (4.25,-0.3) {$v$};

 \node at (-1.4,1.75) {$\begin{pmatrix} 1\\ 2\\ 3\\ 4\\ \end{pmatrix}$};
 \node at (1.25,1.75) {$\begin{pmatrix} 1\\ 3\\ 4\\ 2\\ \end{pmatrix}$};
 \node at (2.85,1.75) {$\begin{pmatrix} 1\\ 2\\ 3\\ 4\\ \end{pmatrix}$};
 \node at (5.5,1.75) {$\begin{pmatrix} 1\\ 3\\ 4\\ 2\\ \end{pmatrix}$};

 \end{tikzpicture}
 \caption{Example of a graph $T$ induced by $\{u,v\}$ and a fractional packing of a $4$-fold cover of $G \setminus T$ that is not extendable to $T$. Each of the colorings $(1,1,1,1), (2,3,2,3), (3,4,3,4),(4,2,4,2)$ of the neighborhood of $T$ occurs with probability $1/4$.}\label{fig:nonextendablefractionalpacking}
\end{figure}
\section{Cycles, degrees, Cartesian product and subdivisons}\label{sec:varia}
In this section we illustrate Lemma~\ref{lem:technicalinduction} with some small applications.
First of, we recover our local-degree result from~\cite{CCDK23}, whose initial proof could be considered a rudimentary form of~\cref{lem:technicalinduction}.
\begin{cor}
Let $G$ be a graph with a correspondence-cover $\sH=(L,H)$ such that $|L(v)|\geq \deg(v)+1$, for every vertex $v$ of $G$. Then $\sH$ has a fractional packing. In particular $\chi_c^{\bullet}(G)$ is bounded from above by $1$ plus the maximum degree of $G$.
\end{cor}
\begin{proof}
By induction on the number of vertices. Choose a vertex $v$ whose list is at least as large as the lists of its neighbors. Apply Lemma~\ref{lem:technicalinduction} to the induced subgraph $T=G\setminus\{v\}$.
\end{proof}
We do not yet know whether $\chi_c^{\bullet}(G)\leq 4$ for bipartite planar graphs $G$. Initial reductions and discharging arguments lead one to suspect that a hypothetical counterexample has many four-faces and (almost) all vertices of degree $3$ or $4$. Could a grid-like graph such as the Cartesian product $C \square P$ of a cycle and a path provide such a counterexample? The next Corollary, which is reminiscent of tight bounds for list coloring of Cartesian products~\cite{BJKM06}, shows that the answer is no.
\begin{cor}
For every graph $G$ and every tree $F$, 
$$\chi_c^{\bullet}(G \square F ) \leq \chi_c^{\bullet}(G)+1.$$
\end{cor}
\begin{proof}
We proceed by induction on the number of vertices of $F$. Let $v$ be a leaf of $F$. Let $T$ be the induced subgraph of $G\square F$ that corresponds to $G \square \{v\}$. By induction, the graph $G\square (F\setminus\{v\})$ has fractional correspondence packing number at most $\chi_c^{\bullet}(G)+1$. Since each vertex of $T$ has precisely one neighbor outside $T$, Corollary~\ref{cor:nosparselyconnectedsubgraph} implies that also $\chi_c^{\bullet}(G \square F ) \leq \chi_c^{\bullet}(G)+1.$
\end{proof}

Subdivided graphs form another easy application.
\begin{cor}\label{cor:chicbull_subdiv}
Let $G$ be obtained from another graph $H$ by subdividing each edge once. Then $\chi_c^{\bullet}(G)\leq 3 $. 
\end{cor}
\begin{proof}
Apply Corollary~\ref{cor:nosparselyconnectedsubgraph} with $r=3$ and $\G$ the class of all once-subdivided graphs. In a minimum counterexample $G$, any star $T$ that is centered at a vertex of $H$ yields a contradiction.
\end{proof}

We can extend this corollary to determine all exact bounds for the packing numbers of subdivisions.

\begin{prop}\label{prop:subdivision}
 Let $G$ be the subdivision of a graph $H$.
 Then $ \chi_\ell^\bullet (G) \le \chi_c^\bullet(G) \le 3$, $\chi_c^\star(G) \le 4,$ $\chi_\ell^\star(G)\le 3$.
 All of these are sharp.
\end{prop}

\begin{proof}
 All claimed bounds are sharp, as they are attained by even cycles.
 Because subdivisions are $2$-degenerate,~\cref{lem:degeneracy} immediately implies $\chi_c^{\star}(G)\leq 4$.
 By~\cref{cor:chicbull_subdiv}, it remains to prove $\chi_\ell^\star(G)\le 3$. 
 Note that $G$ is bipartite with parts $V(H)$ and $V(G)\setminus V(H)$. The vertices in $V(G)\setminus V(H)$ all have degree two, and so by Lemma 33 in~\cite{CCDK21}, $\chi_\ell^\star(G)\le 3$.
\end{proof}

We now turn our attention to fractional packing of cycles. A quick application of Lemma~\ref{lem:monotonicityfractionalpacking} yields an almost-optimal result.
\begin{cor}\label{cor:fracpackcycles_weak}
Let $\sH$ be a correspondence-cover of a cycle $C$. Suppose all lists have size $\geq 2$, and $C$ has two pairs of adjacent vertices with list-size $\geq 3$. Then $\sH$ has a fractional packing.
\end{cor}
\begin{proof}
Let $a,b,c,d$ be the vertices with lists of size $\geq 3$, with $ab$ and $cd$ edges of the cycle, and possibly $b=c$.
By Lemma~\ref{lem:monotonicityfractionalpacking} we may assume that $a,b,c,d$ have lists exactly $3$ and that the remaining lists have size exactly $2$. Without loss of generality we can partition $C$ into a path $P$ with endpoints $a,d$ and another path with endpoints $b,c$. We apply Lemma~\ref{lem:technicalinduction} to $T=P$, which is allowed because for every path, the fractional correspondence packing number is at most two. 
\end{proof}

While we're at it, we expand on~\cref{cor:fracpackcycles_weak} by fully characterizing the correspondence-covers of a cycle that admit a fractional packing. As an ingredient for the induction proof, we use that a degree-$2$ vertex can almost always be suppressed:

\begin{lem}\label{lem:suppressingdegree2}
 Let $G$ be a graph with a degree-$2$ vertex $v$ that has non-adjacent neighbors $u$ and $w$. Let $\sH=(L,H)$ be a correspondence-cover of $G$ with list sizes such that $(|L(u)|,|L(v)|,|L(w)|)$ is either $(2,2,2),(2,3,2)$ or $(2,2,3)$ 
 Suppose furthermore that every correspondence-cover $(L_0,H_0)$ of $G\setminus v+uw$ with $|L_0(z)|=|L(z)|$ for every $z\in V(G)-v$ has a fractional packing. Then $\sH$ has a fractional packing.
\end{lem}

\begin{proof}
Let $L(u)=\{1_u,2_u\}$. Furthermore, let $L(v)=\{1_v,2_v\}$ or $L(v)=\{1_v,2_v,3_v\}$, and $L(w)=\{1_w,2_w\}$ or $\{1_w,2_w,3_w\}$, depending on the size of the lists of $v$ and $w$.

Suppose first that $|L(v)|=2$. Then without loss of generality we may assume that, in $H$, the matchings between the lists are such that $1_u1_v1_w$ and $2_u2_v2_w$ are paths.
We now construct a suitable cover $(L_0, H_0)$ of $G_0:=G\setminus v+uw$. The list-assignment $L_0$ is simply equal to $L$, albeit restricted to the vertices of $G\setminus v$. The new cover graph $H_0$ shall be $H-L(v)$ plus two additional edges, namely $1_u2_w$ and $2_u1_w$. 
By assumption $\sH_0$ has a fractional packing. 
This guarantees a random independent transversal $I_0$ of $\sH_0$, containing either $u_1$ or $u_2$, both with probability half. 
If $1_u\in I_0$ then $I_0$ avoids $\{2_u,2_w\}$, so we can extend $I_0$ to an independent transversal of $\sH$ by adding color $2_v$ from $L(v)$. Otherwise $2_u\in I_0$ so $I_0$ avoids $\{1_u,1_w\}$ so we choose color $1_v$. Note that both $1_v$ and $2_v$ then both occur with probability half $=1/|L(v)|$, as desired.\\

Next, we consider the case $|L(v)|=3$.
If the cover $\sH$ is such that (up to relabeling) $1_u1_v1_w$ and $2_u2_v2_w$ are paths as before, then we construct the same cover $\sH_0$ as before, in which by induction we find an independent transversal $I_0$. 
If $1_u\in I_0$ then $I_0$ avoids $\{2_u,2_w\}$, so we can extend $I_0$ to an independent transversal of $\sH$ by adding color $2_v$ with probability $2/3$ and color $3_v$ with probability $1/3$. 
Otherwise $2_u\in I_0$ so $I_0$ avoids $\{1_u,1_w\}$, so we can extend $I_0$ to an independent transversal of $\sH$ by adding color $1_v$ with probability $2/3$ and color $3_v$ with probability $1/3$. Note that each of $1_v,2_v,3_v$ then occur with probability $1/3=1/|L(v)|$, as desired.\\

What remains is the technically unpleasant case that (up to relabeling), $1_u1_v$, $2_u2_v2_w$ and $3_v1_w$ are paths in $H$.
In this situation we construct two covers of $G_0$, namely $(L_{\alpha},H_{\alpha})$ and $(L_{\beta},H_{\beta})$. Here $H_{\alpha}$ is $H-L(v)$ plus extra edges $1_u2_w$ and $2_u1_w$, while $H_{\beta}$ is $H-L(v)$ plus extra edges $2_u2_w$ and $1_u1_w$. The induction hypothesis gives us suitable random independent transversals $I_{\alpha}$ of $\sH_{\alpha}$ and $I_{\beta}$ of $\sH_{\beta}$. We will now extend these to random independent transversals $I_{\alpha}^+$ and $ I_{\beta}^+$ of $\sH$.

With probability half, $1_u,1_w \in I_{\alpha}$, in which case we extend $I_{\alpha}$ with $2_v$ to obtain $I_{\alpha}^+$
Otherwise (also with probability half) $2_u,2_w\in I_{\alpha}$, so we can extend $I_{\alpha}$ by choosing each of $1_v,3_v$ with equal probability. Thus in total, $1_v,2_v,3_v$ appear in $I_{\alpha}^+$ with probabilitites $1/4, 1/2, 1/4$.

On the other hand, to extend $I_{\beta}$, we do the following. With probability half we have $1_u,2_w \in I_{\beta}$, in which case we extend $I_{\beta}$ with $3_v$.
Otherwise $2_u,1_w \in I_{\beta}$ and we extend $I_{\beta}$ with $1_v$. Thus $1_v,2_v,3_v$ appear in $I_{\beta}^+$ with probabilitites $1/2, 0, 1/2$.

Finally, we choose an aggregated random independent transversal $I^+$ of $\sH$ by selecting $I_{\alpha}^+$ with probability $2/3$ and $I_{\beta}^+$ with probability $1/3$. This ensures that each of $1_v,2_v,3_v$ appear in $I^+$ with probability $1/3=1/|L(v)|$, as desired.
\end{proof}

By first examining cycles of length (at most) six and then applying a straightforward induction, it can be shown that for every cycle $C_n$ of length $n\geq 3$ and for every function $s:V(C_n)\rightarrow \{2,3\}$ such that $s(v)=2$ for all but at most two vertices $v$, there exists a correspondence-cover $\sH=(L,H)$ of $C_n$ with list sizes $|L(v)|=s(v)$ for every $v$, \emph{which does not admit a fractional packing}. 
Hence the following corollary is best possible. 
\begin{cor}\label{cor:fracpackcycles}
Let $\sH$ be a correspondence-cover of a cycle $C$. Suppose all lists have size $\geq 2$ and at least three lists have size $\geq 3$. Then $\sH$ has a fractional packing.
\end{cor}
\begin{proof}

By Lemma~\ref{lem:monotonicityfractionalpacking}, we may assume that precisely three lists have size $3$ and the remaining lists have size $2$.
By repeated application of Lemma~\ref{lem:suppressingdegree2}, we may assume that no three subsequent vertices have list sizes $(2,2,2), (2,3,2)$ or $(2,2,3)$. It follows that $C$ is either a triangle with list-sizes $(3,3,3)$ or a four-cycle with list-sizes $(3,3,3,2)$, so that $\sH$ has a fractional packing by~\cref{cor:fracpackcycles_weak}.
\end{proof}

\section{Fractional packing number of planar graphs with girth at least $6$}\label{sec:planargirth6}
In this section we use discharging and~\cref{cor:nosparselyconnectedsubgraph} applied to trees to prove that planar graphs with girth $\geq 6$ have fractional correspondence packing number at most $3$.

\begin{lem}\label{lem:findingT}
 Every $2$-connected graph $G$, with average degree less than $3$ and at least $2$ vertices of degree $2$, contains an induced subtree $T$ for which every vertex has precisely one neighbor outside $T.$
\end{lem}
\begin{proof}
 Assume the contrary. Note that $G$ has minimum degree $2$.
 For every vertex $v$, let $H_v$ be the largest possible connected (induced) subgraph containing $v$ such that all vertices of $H_v\setminus \{v\}$ have degree\footnote{In this section, degrees always refer to the degree in $G$.} at most $3$.
 \begin{claim}\label{clm:degv-2}
 $H_v\setminus \{v\}$ contains at most $\deg(v)-2$ vertices of degree $2$.
 \end{claim}
 \begin{claimproof}
 When $\deg(v)=2,$ and $u \in H_v$ where $\deg(u)=2,$ there is a shortest path between $u$ and $v$ in $H_v.$ This path only contains vertices of degree at most $3$ and hence has a (induced) subpath starting and ending in degree-$2$-vertices and having degree-$3$-vertices inside. This would be a valid subtree $T$, contradiction. Thus the claim is true for $\deg(v)=2.$
 When $\deg(v)\ge 3,$ we note that $H_v \setminus \{v\}$ contains at most $\deg(v)$ components, which all contain at most one degree-$2$-vertex. 
 If there are $\geq \deg(v)-1$ components with a degree-$2$-vertex, we can take shortest paths between $v$ and each of these degree-$2$-vertices. The union is a tree $T$, all of whose vertices have degree $1$ outside $T$, contradiction. So $H_v$ has at most $\deg(v)-2$ vertices of degree $2.$
 \end{claimproof}
 Now we do discharging, where every vertex starts with charge $\omega(v)=\deg(v).$
 Every vertex $v$ of degree at least $4$, gives a charge of $\frac 12$ to every vertex in $H_v$ of degree $2.$
 Then the final charge of $v$ is at least $\deg(v)-\frac12 \left( \deg v -2\right)=\frac 12 \deg v +1 \ge 3$ by~\cref{clm:degv-2}.
 Since $G$ is $2$-connected, for a degree $2$-vertex $v$, $H_v$ has at least $2$ neighbors in $G$ outside $H_v$ (consider two internally disjoint paths from $v$ towards another degree $2$-vertex, which must lie outside the closed neighborhood of $H_v$). Every neighbor $w$ of $H_v$ outside $H_v$ has degree at least $4$, and by definition $v \in H_w$. This implies that $v$ receives at least twice a charge equal to $\frac12$ and as such ends with a charge which is at least $3.$
 Since every vertex ends with a charge of at least $3$, the initial average degree had to be at least $3$, from which the conclusion follows.
\end{proof}

\begin{thr}\label{thr:g6_chicbull_le3}
 Let $G$ be a planar graph with girth at least $6,$ then $\chi_c^\bullet(G)\le 3.$
\end{thr}
\begin{proof}
 Note that a minimum counterexample has no cutvertex and thus is $2$-connected and has minimum degree $2.$ By Euler's formula for planar graphs, we know that $G$ has average degree less than $3$, and at least $6$ vertices of degree $2$.
 Since the graph class of planar graphs with girth at least $6$ is closed under taking subgraphs, by combining \cref{lem:findingT} and \cref{cor:nosparselyconnectedsubgraph}, no vertex-minimum graph with $\chi_c^\bullet(G)>3$ exists, implying the theorem.
\end{proof}

Say a graph is \emph{almost-cubic} if it is $2$-connected, one vertex has degree $2$ and the other vertices have degree $3$. Every minimum degree $2$ average-degree $<3$ graph which is not almost-cubic must contain at least two degree-$2$ vertices. Hence by the same proof, Theorem~\ref{thr:g6_chicbull_le3} generalizes to
\begin{thr}\label{thm:mad<3}
Let $G$ be a graph with $\mad(G)<3$, such that no induced subgraph is almost-cubic. Then $\chi_c^{\bullet}(G)\leq 3$.
\end{thr}

\begin{rem}
 Let $H$ be $K_{2,3}$ plus an edge in the partition class of $3$ vertices. This is an example of an almost-cubic graph
 %with $\sum_v \deg(v)=3n-1$
 without any induced subtree as described in~\cref{lem:findingT}, showing its sharpness. In fact, $\chi_c^{\bullet}(H)>3$, showing that the extra condition in~\cref{thm:mad<3} is necessary as well. However, since $\mad(H)=\frac{14}{5}$ and there is no almost-cubic graph on fewer than five vertices, $\mad(G)<\frac{14}{5}$ is a sufficient condition for $\chi_c^{\bullet}(G)\leq 3$.
\end{rem}

\section{$K_4$-minor-free graphs}\label{sec:K4minorfree}

This section is devoted to the optimal bounds for $K_4$-minor-free graphs. As mentioned in the introduction these are the graphs whose blocks are series-parallel.
We consider the general and triangle-free case separately.

\subsection{Series-parallel graphs (no girth condition)}
 The following theorem provides the desired characterisation.
\begin{thr}\label{thm:series-parallel}
 For every series-parallel graph $G$
 \begin{itemize}
 \item $\chi_\ell^\star(G), \chi_c^\bullet(G), \chi_c^\star(G) \le 4,$
 \item $\chi_\ell^\bullet(G)\le 3.$
 \end{itemize}
 All of these inequalities are sharp and attained by infinitely many 2-connected outerplanar graphs.
\end{thr}

Series-parallel graphs have treewidth bounded by $2$, so the degeneracy $\delta^{\star}$ of a series-parallel graph is bounded by $2$. By~\cref{lem:degeneracy}, both list packing number and correspondence packing number are bounded by $2\delta^\star=4$. As pointed out in~\cite{CCDK23}, the fan $F_7$ on seven vertices is an example of an outerplanar series-parallel graph which attains these bounds, i.e. $\chi_\ell^\star(F_7)= \chi_c^\star(F_7)=4.$ The same is true for any larger fan.
Thus only the fractional packing numbers in~\cref{thm:series-parallel} are nontrivial. 
Since the maximal series-parallel graphs are exactly the $2$-trees, it is sufficient to consider these.
In~\cite{BMS22} it was shown that every $3$-list-assignment $L$ of a $2$-tree admits six $L$-colorings such that every vertex $v$ is colored $c$ by precisely two of these colorings, for every $c\in L(v)$. Taking each of these $L$-colorings with probability $1/6$, it follows that every $2$-tree has fractional list packing number at most $3$.

\begin{prop}[\cite{BMS22}, Theorem 3.2]~\label{prop:2treeslist}
 For every $2$-tree, $\chi_\ell^\bullet(G)\le 3.$
\end{prop}

To complete the proof of Theorem~\ref{thm:series-parallel}, it remains to demonstrate that Proposition~\ref{prop:2treeslist} does not survive in the correspondence setting. As it turns out, it already fails for a small outerplanar graph (and hence also for every other outerplanar $2$-tree containing it).

\begin{prop}\label{prop:$2$-tree}
 There exists an outerplanar $2$-tree $G$ for which $\chi_c^\bullet(G)=4.$
\end{prop}
\begin{proof}
Let $G$ be the six-vertex graph obtained from a triangle $x_1x_2x_3$ by adding a vertex $x_{ij}$ and making it adjacent to $x_i$ and $x_j$, for every $i\neq j$. Consider the correspondence-cover $(H,L)$ of $G$ depicted in Figure~\ref{fig:outerplanar_cover}, which is a subcover of a $3$-fold cover of $G$. We need to show that the fractional chromatic number of $H$ is strictly larger than three. By linear programming duality, it suffices to find a fractional clique of total weight larger than $3$. I.e., we need to find an assignment of weights $w:V(H)\rightarrow [0,1]$ such that $\sum_{v\in V(H)} w(v) >3$, and $\sum_{v \in S} w(s) \leq 1$ for every independent set $S$ of $H$. The assignment in Figure~\ref{fig:outerplanar_cover} is such an assignment with total weight $22/7>3$.
\end{proof}

\begin{figure}[h!]
\centering
\tikzmath{\base=6.5; \dt = 0.75; \sqf=0.866025; \x= \sqf*\dt; \y= \dt/2; }
\begin{tikzpicture}

 \draw[fill=gray!17, rotate around={90:(0,\dt)}] (0,\dt) ellipse (40pt and 15pt);
 \draw[fill=gray!17, rotate around={90:(0,2*\sqf*\base-\dt)}] (0,2*\sqf*\base-\dt) ellipse (40pt and 15pt); 
 \draw[fill=gray!17, rotate around={30:(\base/2-\x,\sqf*\base-\y)}] (\base/2-\x,\sqf*\base-\y) ellipse (40pt and 15pt); 
 \draw[fill=gray!17, rotate around={-30:(-\base/2+\x,\sqf*\base-\y)}] (-\base/2+\x,\sqf*\base-\y) ellipse (40pt and 15pt); 
 \draw[fill=gray!17, rotate around={-30:(\base-\x,\y)}] (\base-\x,\y) ellipse (40pt and 15pt);
 \draw[fill=gray!17, rotate around={30:(-\base+\x,\y)}] (-\base+\x,\y) ellipse (40pt and 15pt); 

\definecolor{weight0}{rgb}{1,0.4,0}
\definecolor{weight1}{rgb}{1,1,1}
\definecolor{weight2}{rgb}{0,0,0}
\begin{scope}[every node/.style={circle,thick,draw}]
 \node (0) [fill=weight2] at (\base/2-2*\x,\sqf*\base-2*\y) {};
 \node (1) [fill=weight2] at (\base/2-\x,\sqf*\base-\y) {};
 \node (2) [fill=weight2] at (\base/2,\sqf*\base) {};
 \node (3) [fill=weight1]at (0,2*\dt) {};
 \node (4) [fill=weight1] at (0,\dt) {};
 \node (5) [fill=weight2] at (0,0) {} ;
 \node (6) [fill=weight2] at (-\base/2+2*\x,\sqf*\base-2*\y) {} ;
 \node (7) [fill=weight2] at (-\base/2+\x,\sqf*\base-\y) {};
 \node (8) [fill=weight2] at (-\base/2,\sqf*\base) {};
 \node (9) [fill=weight1] at (0,2*\sqf*\base-2*\dt) {};
 \node (10) [fill=weight1] at (0,2*\sqf*\base-\dt) {};
 \node (12) [fill=weight1] at (\base-2*\x,2*\y) {};
 \node (14) [fill=weight1] at (\base,0) {} ;
 \node (16) [fill=weight1] at (-\base+\x,\y) {} ;
 \node (17) [fill=weight1] at (-\base,0) {} ; 
\end{scope}

\begin{scope}[
 every node/.style={fill=white,circle},
 every edge/.style={draw=black,very thick}]
 \path (0) edge (3);
 \path (1) edge (4);
 \path (2) edge (5); 
 \path (0) edge (3);
 \path (0) edge (6);
 \path (2) edge (7);
 \path (1) edge (8);
 \path (3) edge (7);
 \path (4) edge (6);
 \path (5) edge (8);
 \path (9) edge (0);
 \path (9) edge (6);
 \path (10) edge (1);
 \path (10) edge (7);
 \path (12) edge (1);
 \path (12) edge (3);
 \path (14) edge (2);
 \path (14) edge (5);
 \path (16) edge (4);
 \path (16) edge (7);
 \path (17) edge (5);
 \path (17) edge (8);
\end{scope}
\end{tikzpicture}
 \caption{A correspondence-cover $(H,L)$ of a six-vertex outerplanar $2$-tree $G$, demonstrating that $\chi_{c}^{\bullet}(G) >3$. The cover has three size-$3$ lists and three size-$2$ lists, which are indicated by grey ellipses. To avoid clutter, the edges of the cliques on the lists are not drawn in this figure. A fractional clique of $H$ with total weight $22/7>3$ is obtained by assigning every black vertex weight $2/7$ and every white vertex weight $1/7$. 
 Due to the cliques on the six lists, every independent set of $H$ contains at most six vertices, at most three of which are black. 
 Every independent set with three black vertices has at most one other (white) vertex.
 Every independent set with two black vertices has at most three other (white) vertices. 
 Therefore every independent set has weight at most $1$, as required for a fractional clique.}
 \label{fig:outerplanar_cover}
\end{figure}
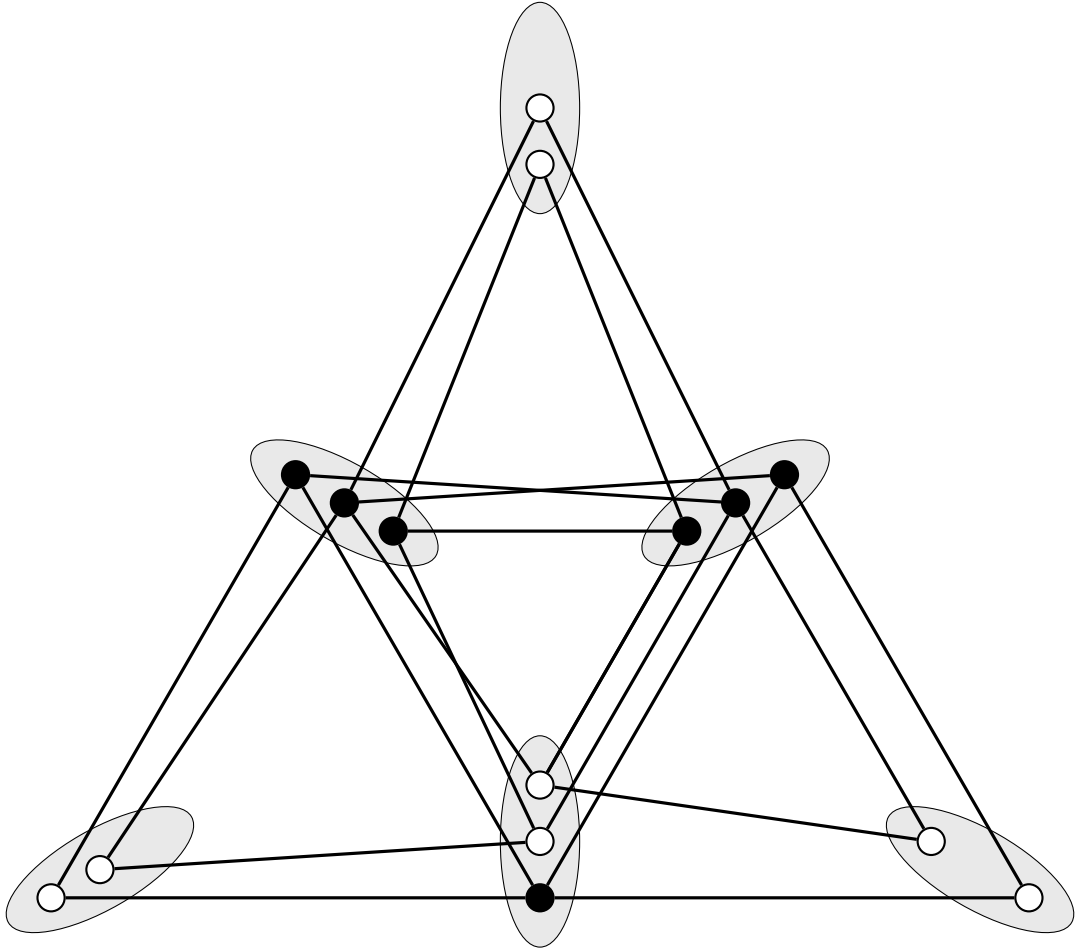

\subsection{Series-parallel graphs with girth at least $4$}

In this subsection, we prove that if a series-parallel graph has girth at least $4$, that $\chi_c^\star(G)\le 4$ (equality by cycles) and $\chi_\ell^\bullet(G), \chi_\ell^\star(G), \chi_c^\bullet(G)\le 3.$
The two essential cases are to prove that $\chi_\ell^\star(G)$ and $ \chi_c^\bullet(G)$ are bounded by $3.$

Both proofs will be inductively and use the definition of two-terminal series-parallel graphs.
The class of two-terminal series-parallel graphs is defined as follows:
\begin{enumerate}
 \item $(K_2;x,y)$ is a two-terminal series-parallel graph, where $x,y$ are the two end vertices of the edge $K_2$.
 \item If $(G_1;x_1,y_1)$ and $(G_2;x_2,y_2)$ are two-terminal series-parallel graphs, then the series join $(G;x_1,y_2)$ obtained from the disjoint union of $G_1$ and $G_2$ by identifying $y_1$ with $x_2$ is a two-terminal series-parallel graph.
 \item If $(G_1;x_1,y_1)$ and $(G_2;x_2,y_2)$ are two-terminal series-parallel graphs, then the parallel join $(G;x,y)$ obtained from the disjoint union of $G_1$ and $G_2$ by identifying $x_1$ with $x_2$ into a new vertex $x$, $y_1$ and $y_2$ into a new vertex $y$ is a two-terminal series-parallel graph.
\end{enumerate}

The hard part of the induction step is in the parallel case for $\chi_\ell^\star$, while the hard case for $\chi_c^\bullet$ is the series case (due to the strengthened induction hypothesis and thus adding a parallel edge in certain cases).
We will also use the following elementary lemma.

\begin{lem}\label{lem:extend_prepacking_Cn}
 Let $L$ be a $3$-list-assignment of $C_n.$
 Then any partial packing of $2$ neighbors $x,y$ with the same list can be extended to a packing of $C_n.$
\end{lem}
\begin{proof}
Let the vertices of the cycle be $x,y,z, z_4, \ldots, z_n.$
If $n=3$, the statement is immediate; take the common derangement of $\vec{c}(x)$ and $ \vec{c}(y)$ and switch colors if the list of $z$ has some other colors.
Now let $n \ge 4$. If a neighbor of $x$ or $y$ has the same list as them, w.l.o.g. $z$, then the packing with $\vec c(z)=\vec c(x)$ can be extended. Indeed, we can contract the edges $xy$ and $yz$ to obtain a reduced graph isomorphic to $K_2$ or a cycle $C_{n-2}$ with a partial packing of only one vertex $z$, given by $\vec c(z)$. Because $\chi_{\ell}^{\star}(K_2), \chi_{\ell}^{\star}(C_{n-2}) \le 3$, the reduced graph has a packing, which by relabeling the colorings can be made to agree with $\vec{c}(z)$. Lifting this packing back to $C_n$ yields the desired extension.
Otherwise the lists of $x$ and $y$ differ from the lists of their neighbors $z$ and $z_n$. In that case, starting from $y$, we have at least $3$ possible choices for the permutation $\vec c(z)$ and importantly they do not all have the same parity. Hence for all $i\geq 4$ there are at least $5$ possible choices for $\vec c(z_i)$ such that $\vec{c}(\{x,y, z_i\})$ extends to a packing of the subpath $x,y,z,z_1,\ldots,z_i$.
Since at most three choices of $\vec c(z_n)$ are incompatible with $\vec c(x)$, there is a valid choice for $\vec c(z_n)$ to complete the packing.
\end{proof}

The next proposition implies that $\chi_{\ell}^{\star}(G)\leq 3$ for series-parallel graphs of girth at least $4$.
\begin{prop}
\label{prop-sp}
 If $(G;x,y)$ is a two-terminal series-parallel graph of girth at least $4$, and $L$ is a $3$-list-assignment of $G$, then there is an $x$-$y$-path $P$ of length equal to the distance between $x$ and $y$ in $G$ and a 3-list-assignment $L'$ of $P$ with $L'(x)=L(x), L'(y)=L(y)$ such that for any $L'$-packing $\phi$ of $P$, the restriction of $\phi$ to $x$ and $y$ can be extended to an $L$-packing of $G$. 
\end{prop}
\begin{proof}
 The proof is by induction on the number of vertices of $G$. If $G=K_2$, then $P=G$ and $L'=L$. The conclusion is trivially true. 

 Assume $|V(G)| \ge 3$. Then there are two-terminal graphs $(G_1;x_1,y_1)$ and $(G_2;x_2,y_2)$ such that $(G;x,y)$ is the series join or the parallel join of $(G_1;x_1,y_1)$ and $(G_2;x_2,y_2)$. 
 
 By the induction hypothesis, for $i=1,2$, there is an $x_i$-$y_i$-path $P_i$ in $G_i$ of length $\text{dist}_{G_i}(x_i,y_i)$ and a 3-list-assignment $L'_i$ such that 
 for any $L'_i$-packing $\phi_i$ of $P_i$, the restriction of $\phi_i$ to $x_i$ and $y_i$ can be extended to an $L$-packing of $G_i$. 
 
If 
$(G;x,y)$ is the series join of $(G_1;x_1,y_1)$ and $(G_2;x_2,y_2)$, then let $P$ be obtained from the disjoint union of $P_1,P_2$ by identifying $y_1$ and $x_2$, and let $L'$ be the union of $L'_1$ and $L'_2$. It is obvious that the length of $P$ equals the distance between $x$ and $y$ in $G$ and for any $L'$-packing $\phi$ of $P$, the restriction of $\phi$ to $x$ and $y$ can be extended
 to an $L$-packing of $G_1$ and $G_2$, whose union is an $L$-packing of $G$. 

Assume $(G;x,y)$ is the parallel join of $(G_1;x_1,y_1)$ and $(G_2;x_2,y_2)$. We may assume that $\ell(P_1) \ge \ell(P_2)$. Since $G$ has girth at least $4$, one of the following cases occurs.

\textbf{Case 1:} $L(x)=L(y)$. 

Let $P$ be an $x$-$y$-path of length $\ell(P_2)$ and $L'(v)=L(x)$ for each vertex $v$ of $P$. Assume $\phi$ is an $L'$-packing of $P$. We need to show that the restriction of $\phi$ to $x$ and $y$ can be extended to an $L'$-packing of $G$. Note that $\phi(x)$ and $\phi(y)$ are permutations of $L(x)$ of the same parity by the choice of $L'.$

 \begin{claim}\label{claim:extendabilityofpath}
 For any path $P''$ of length at least $\ell(P)$ connecting $x$ and $y$, for any 3-list-assignment $L''$ of $P''$ with $L''(x)=L(x)=L(y)=L''(y)$, the restriction of $\phi$ to $x$ and $y$ can be extended to an $L''$-packing of $P''$.
 \end{claim} 
 \begin{claimproof}
 If $P''$ has length $1$, then $P=P''$ and $L''=L'$. Hence $\phi$ is an $L''$-packing of $P''$. Assume $\ell(P'') \ge 2$. 
 If $\phi(x)=\phi(y),$ then we identify $x$ and $y$, creating a $K_2$ (when $P''$ has length $2$) or $C_{n}$ (when $P''$ has length $n \ge 3$).
 Now, the conclusion follows since $\chi_{\ell}^\star(K_2) <3$
 and $\chi_{\ell}^\star(C_n)=3$. 
If $\phi(x)\not=\phi(y),$ then we add an edge between $x$ and $y$, creating a cycle $C_{\ell(P'')+1}$ and the conclusion follows from \cref{lem:extend_prepacking_Cn}.
 \end{claimproof}
 It follows from the claim that the restriction of $\phi$ to $x$ and $y$ can be extended to an $L'_i$-packing of $P'_i$ for $i=1,2$ and hence can be extended
 to an $L$-packing of $G_1$ and $G_2$, whose union is an $L$-packing of $G$.

\textbf{Case 2a:} $\ell(P_1) \ge 3$ and $L(x) \ne L(y)$. 

By a case analysis (similar to the end of the proof of~\cref{lem:extend_prepacking_Cn}), it can be shown that any $L'_1$-packing of $x_1$ and $y_1$ can be extended to an $L'_1$-packing of $P_1$. Let $P=P_2$ and $L'=L'_2$. Then 
 the length of $P$ equals the distance between $x$ and $y$ in $G$ and for any $L'$-packing $\phi$ of $P$, the restriction of $\phi$ to $x$ and $y$ can be extended
 to an $L$-packing of $G_1$ and $G_2$, whose union is an $L$-packing of $G$. 

\textbf{Case 2b:} $\ell(P_1) =\ell(P_2)=2$ and $L(x) \ne L(y)$.

If $\abs{L(x) \cap L(y)}\le 1,$ then $\phi_1$ and $\phi_2$ always exist by Hall's Matching theorem.
In the other situation, assume without loss of generality that $L(x)=\{1,2,3\}$ and $L(y)=\{1,2,4\}$.
A list which does not contain both $1$ and $2$ can always be permuted into a derangement of $\vec c(x)$ and $ \vec c(y)$ by Hall's Matching theorem.
If at least one of the centers of $P_1$ or $P_2$ has a list containing both $1$ and $2$, take that one (with a preference of a list which is $\{1,2,3\}$ or $\{1,2,4\}$ if possible), wlog $P_1$, as the list-assignment of the center of path $P$ of the parallel composition.
Again by Hall's matchings theorem, also $\phi_2$ exists.
\end{proof}

Next, we prove that a (two-terminal) series-parallel graph $G$ of girth at least $4$ satisfies $\chi_{c}^{\bullet}(G) \le 3.$
We prove this using a stronger induction hypothesis.

\begin{prop}\label{prop:chicbullet_sp_g4}
 Let $(G;x,y)$ be a two-terminal series-parallel graph of girth at least $4$.

 \begin{enumerate}
 \item\label{itm:ge2} 
 If $d(x,y) \ge 2$, for every full $3$-correspondence-cover $(L,H)$ of $G$, there exists a fractional packing $c$ of $(L,H)$ such that every possible partial coloring $(c(x), c(y))$ appears with probability $\frac 19$.
 \item\label{itm:not2} If $d(x,y) \not=2,$ we consider $G \cup \{xy\}$, the graph $G$ with the edge $xy$ added if it is not part of $G$ already.
 For every full $3$-correspondence-cover $(L,H)$ of $G \cup \{xy\}$, there exists a fractional packing $c$ of $(L,H)$ such that every possible proper partial coloring $(c(x), c(y))$
 appears with probability $\frac 16$. 
 \end{enumerate}
\end{prop}

\begin{proof}
 The proof is again by induction on the number of vertices of $G$. If $G=K_2$, then the conclusion is trivially true. 

 Assume $|V(G)| \ge 3$. Then there are two-terminal graphs $(G_1;x_1,y_1)$ and $(G_2;x_2,y_2)$ such that $(G;x,y)$ is the series join or the parallel join of $(G_1;x_1,y_1)$ and $(G_2;x_2,y_2)$. 

 If $(G;x,y)$ is the series join of $(G_1;x_1,y_1)$ and $(G_2;x_2,y_2)$, we prove that the statements are true by induction.

 We first prove~\cref{itm:ge2}. There are three cases to consider depending on the information about the fractional packings of $(G_1;x_1,y_1)$ and $(G_2;x_2,y_2)$.

 \textbf{Case 1: $d(x_1,y_1)\ge 2$ and $d(x_2,y_2)\ge 2$}

 We consider each of the $27$ possibilities for $(c(x_1),c(x_2),c(y_2))$ with equal probability (a fraction equal to $\frac1{27}$) and extend the restrictions to $G_1$ and $G_2$ to fractional packings.
 Then we have a fractional packing of $G$ for which the condition in~\cref{itm:ge2} is satisfied.

 One can formalize this by considering the independent transversals of a fixed cover $(L, H)$ of $G$. Let $(L_i,H_i)$ be the cover $(L,H)$ restricted to $G_i$, for $i\in\{1,2\}$.
 A fractional correspondence-coloring of $G_i$ is a probability distribution $f_i$ on independent transversals of $(L_i, H_i)$ such that each pair $(x,v)$, $x \in L_i(v), v \in V(G_i)$ appears with probability $\frac 13.$
 % The same notions are defined for $G_2.$
 For any two independent transversals $I_1 \in (L_1, H_1), I_2 \in (L_2,H_2),$ let $f(I_1 \cup I_2)= 3f_1(I_1)f_2(I_2)$ if $I_1 \cap L_1(y_1)=I_2 \cap L_2(x_2)$ and $0$ otherwise.
 Then $f$ is a probability distribution on independent transversals of $(L, H)$, for which every combination of two vertices in $L(x_1) \times L(y_2)$ appears with the same probability (equal to $\frac 19$).

 \textbf{Case 2: $d(x_1,y_1)=d(x_2,y_2)=1$}

 Up to relabeling, we can assume that $L(x_1)=L(x_2)=L(y_2)=\{1,2,3\}$ and that the matchings on $x_1y_1$ and $x_2y_2$ are the full identity matchings.
 We now consider a linear combination of twelve possible restrictions of the colorings to $\{x_1, x_2, y_2\}$.
 Each of the $6$ possibilities of $(c(x_1),c(x_2),c(y_2))$ which are permutations of $[3]$ are taken with weight $\frac 19$ and the $6$ possibilities of $(c(x_1),c(x_2),c(y_2))$ where $c(x_1)=c(y_2)\not =c(x_2)$ are taken with weight $\frac{1}{18}.$
 Each of the $6$ possible restrictions to $(c(x_1),c(x_2))$ and $(c(x_2),c(y_2))$ appear with a fraction $\frac 19+\frac{1}{18}=\frac 16$. By the induction hypothesis, they can be extended to fractional packings on $G_1$ and $G_2$.

 Since each combination of $(c(x_1),c(y_2))$ appears with weight $2\cdot \frac{1}{18}=\frac 19,$ the final partial packing satisfies the condition.
 The reader can also verify every probability, using~\cref{tab:psg4_d11}.

 \begin{table}[h]
 \centering
 \begin{tabular}{|c|cccccccccccc|}
 \hline
 $c(x_1)$ & 1 & 1 & 1 & 1 &2 &2 &2 &2 & 3& 3& 3& 3\\
 $c(x_2)$ & 2 & 2& 3& 3 & 1 & 1& 3& 3 &1 & 1 &2 &2 \\
 $c(y_2)$& 1 &3 & 1&2 &2 & 3& 1&2&2 & 3& 1&3 \\
 \hline
 $\text{fraction} $& $\frac 1{18} $& $\frac 19$& $\frac 1{18} $& $\frac 19 $& $\frac 1{18} $& $\frac 19$& $\frac 19 $& $\frac 1{18} $&$\frac 19$& $\frac 1{18} $& $\frac 19 $& $\frac 1{18} $\\
 \hline
 \end{tabular}
 \caption{The combinations of $(c(x_1) , c(x_2) , c(y_2))$ 
 %resp. $(c(y_1), c(x_1),c(y_2))$ 
 that can occur simultaneously for $G$ being a series join of $G_1$ and $G_2$ when $d(x_1,y_1)=1=d(x_2,y_2)$ (item $1$) 
 %or $d(x_1,y_1)=1<d(x_2,y_2)$ (item $2$), 
 and the fraction of appearance of each combination.}
 \label{tab:psg4_d11}
 \end{table}

 \textbf{Case 3: $d(x_1,y_1)=1$ and $d(x_2,y_2)\ge 2$}

 We consider each of the $18$ possibilities for $(c(x_1),c(x_2),c(y_2))$ where $c(x_1)\not =c(x_2)$ with equal fraction of $\frac 1{18}$ and extend the restrictions to $G_1$ and $G_2$ to fractional packings.
 Then we have a fractional packing of $G$ for which the condition in~\cref{itm:ge2} is satisfied.

 Next, we prove~\cref{itm:not2}, for which we only have to deal with two cases.
 
 \textbf{Case 1: $d(x_1,y_1)\ge 2$ and $d(x_2,y_2)\ge 2$}

 For each of the $6$ possible choices for restricted colorings of $(c(x_1),c(y_2))$ for a full $3$-correspondence-assignment of $G \cup \{xy\}$,
 we uniformly consider a combination with the three options for $c(x_2)$. By~\cref{itm:ge2}, these combinations can be extended to fractional packings of $G_1$ and $G_2$. This gives us a fractional packing of $G \cup \{xy\}$ with the desired property.

 \textbf{Case 2: $d(x_1,y_1)=1$ and $d(x_2,y_2)\ge 2$}

 As was the case with Case $2$ in the inductive proof of~\cref{itm:ge2}, there are $12$ combinations for $(c(y_1), c(x_1),c(y_2))$ (triples in~\ref{tab:psg4_d11} in different order) which can be taken with probabilities equal to $\frac 19$ and $\frac 1 {18}$ such that the fractional packings restricted to $G_1$ and $G_2$ can be taken appropriately.

 Finally, assume that $(G;x,y)$ is the parallel join of $(G_1;x_1,y_1)$ and $(G_2;x_2,y_2)$.
 In the case we want to prove~\cref{itm:ge2}, we need to assume that $d(x_1,y_1)\ge 2$ and $d(x_2,y_2)\ge 2$.
 Now we can combine the fractional packings of $G_1$ and $G_2$, which both attain all $9$ possibilities for $(c(x),c(y))$ with uniform probability.

 In the case where we need to prove~\cref{itm:not2}, we need to assume that $d(x_1,y_1),d(x_2,y_2)\not=2.$
 Consider a correspondence-cover $(L,H)$ of $G \cup \{xy\}$, then for both the restrictions $(L_1,H_1)$ and $(L2,H_2)$ of $(L,H)$ to $G_1 \cup \{x_1y_1\}$ and $G_2 \cup \{x_2y_2\}$, by induction (and remembering the girth $4$ constraint) there exists a fractional packing packing attaining every proper combination of $(c(x),c(y))$ with probability $\frac16$ and thus so does $G \cup \{xy\}$.

 By having proven the induction basis and induction step in every case, we conclude that the Proposition is true for every two-terminal series-parallel graph of girth at least $4.$
\end{proof}

\section{$K_{2,4}$-minor-free graphs}\label{sec:K24minorfree}
This section is devoted to $K_{2,t}$-minor-free graphs, for $t\in \{2,3,4\}$. We will restrict to those that are $2$-connected, which does not pose a significant restriction because (fractional) packings can be easily glued along cut vertices. Let us first quickly deal with the small values of $t$. A $K_{2,2}$-minor-free graph has no cycle of length $\geq 4$, so the (fractional) packing numbers are determined by those of a triangle and hence optimally bounded by $3$. By a well-known consequence of
Wagner’s characterization of planar graphs~\cite{Wagner37}, every $2$-connected $K_{2,3}$-minor-free graphs is outerplanar or isomorphic to $K_4$. Since outerplanar graphs are also $K_4$-minor-free,~\cref{thm:series-parallel} and $\chi_c^\star(K_4)=4$ (corollary of e.g.~\cite[Thm.~4]{CCDK23}) imply that all (fractional) packing numbers of $K_{2,3}$-minor-free graphs are optimally bounded by $4$.
The remainder of this section is devoted to % ($2$-connected) 
$K_{2,4}-$minor-free graphs.

Given a pair of two distinct vertices $x,y$, a graph is \emph{$xy$-outerplanar} if $G+xy$ is a block with an outerplane embedding in which $xy$ is on the outer face, so in particular $G+xy$ is outerplanar. %Equivalently, $xy$-outerplanar graphs are precisely the graphs that have no $K_{2,2}$-$ minor where x and y correspond to the two vertices on one side of the bipartition of $K_{2,2}$.
%(See~\cite[Lem.~16]{EMKT16}). 

%We say a graph $G$ is \emph{outer-reducible} to another graph $H$ if $G$ can be obtained from $H$ by replacing some (possible none) of its edges $xy$ with an $xy$-outerplanar graph. In particular $G$ is outerreducible to itself. 
Recall that $K_5^{-}$ denotes the complete graph on five vertices minus an edge.

%We now formulate the main result, focused on $2$-connected $K_{2,4}$-minor-free graphs.

\begin{lem}\label{lem:K24minorfree+}
 Let $G$ be a $K_{2,4}$-minor-free graph. If $G$ does not contain $K_5^-$ as a subgraph, then 
 $\chi_{c}^{\star}(G) \leq 4$.
\end{lem}

\begin{proof}
 Let $G$ be a minimum counterexample. Since the correspondence packing number of a graph equals the maximum among all its blocks, $G$ is 2-connected. 
 The class of $2$-connected $K_{2,4}-$minor-free graphs was characterized in~\cite[Thr.~19]{EMKT16}. By this characterization, $G$ satisfies one of the following
 \begin{enumerate}
 \item $G$ is outerplanar;
 \item $G$ is the union of three $xy$-outerplanar graphs $H_1,H_2,H_3$ and possibly the edge $xy$, such that $|V(H_i)|\geq 3$ for each $i$ and $V(H_i) \cap V(H_j)=\{x,y\}$ for $i\neq j$.
 \item $G$ is obtained from a $3$-connected $K_{2,4}-$minor-free graph $G_0$ by replacing some of its edges $xy$ with an $xy$-outerplanar graph $H_i$, such that $V(H_i)\cap V(G_0)=\{x,y\}$ and $V(H_i) \cap V(H_j) \subseteq V(G_0)$ for $i \neq j$. Here, the edges $xy$ of $G_0$ that can be replaced need to belong to a particular subset \text{subd}$(G_0)$ of so-called subdividable edges.
 \end{enumerate}

By~\cref{lem:degeneracy}, $G$ has minimum degree at least three. Hence $G$ is not an outerplanar graph, and 
not the union of three $xy$-outerplanar graphs $H_1,H_2,H_3$ and possibly the edge $xy$, as these graphs have minimum degree $2$.

In the third case, suppose $G$ is not $3$-connected. Let $H_i$ be one of the $xy$-outerplanar graphs which is not an edge. Then as before, the outerplanarity of $H_i+xy$ implies that it contains a vertex of degree $\leq 2$ that avoids $\{x,y\}$, so $G$ contains a vertex of degree $\leq 2$, contradiction.

Thus $G$ is $3$-connected $K_{2,4}$-minor-free. 
By~\cite[Thr.~12]{EMKT16}, $G$ is either one of the $9$ small graphs in $\mathcal{G}_0:=\{K_5,K_{3,3}, A, A^+,B,B^+,C,C^+,D \}$ depicted in~\cite[Fig.~3]{EMKT16}, or a graph in $$\G = \{G_{n,r,s}, G_{n,r,s}^+: 2 \le r \le s \text{ and } n-2 \le r+s \le n-1 \} \cup \{G_{n,1,n-3}^+: n \ge 4\},$$ where the graph $G_{n,r,s}$ consists of a path $v_1v_2\ldots v_n$ together with edges
$v_1v_{n-i}$ for $1 \le i \le r$ and $v_{i+1}v_n$ for $1\le i \le s$, and 
$G_{n,r,s}^+$ is obtained from $G_{n,r,s}$ by adding edge $v_1v_n$ (see~\cref{fig:3MainSmallexceptions} for an illustration of these graphs). Hence all interior vertices of the path have degree $3$, except possible one.
This implies that every graph of order at least $8$ in $\G$ contains an induced path $P_3$ of degree-$3$ vertices such that after deleting its vertices, the remaining graph is either $2$-degenerate or reduces (by deleting degree-$2$ vertices to a small wheel graph (which has correspondence packing number bounded by $4$ by applying~\cref{lem:chi_c^star>4} again). That is, most graphs of $\G$ are not counterexamples by~\cref{lem:chi_c^star>4}.
There are two graphs of order $6$ and $7$ in this family which have to be handled separately; $G_{6,2,3}^+, G_{7,3,3}^+$ (the result for the subgraph $G_{7,3,3} \subset G_{7,3,3}^+$ is then a corollary). They are depicted in Figure~\ref{fig:3MainSmallexceptions}.
 The remaining graphs in $\mathcal{G}$ have order $\leq 5$ and hence are (see the remark in~\cite[page~957]{EMKT16}) isomorphic to $K_5, K_4, K_5^{-}$ or $W_5$ (the latter being the wheel graph on five vertices).

It is known~\cite[Thr.~4]{CCDK23} that  
$\chi_c^\star(K_4)=4$, and 
 graphs $K_5$ and $K_5^-$ are excluded by the assumption of the lemma. It remains to show that $\chi_c^{\star}(G) \le 4$ for $G \in \{ K_{3,3},A,B, B^+, C, C^+, D , W_5, A^+, G_{6,2,3}^+,G_{7,3,3}^+\}.$

 \begin{claim}
 For every $G \in \{ K_{3,3},A,B, B^+, C, C^+, D , W_5\}$, $\chi_c^{\star}(G)=4$.
 \end{claim}
 \begin{claimproof}
 As illustrated in~\cref{fig:smallgraphs}, the graphs $K_{3,3},A,B, B^+, C, C^+, D , W_5$ all contain an induced $P_3$ or $C_3$ of degree-$3$ vertices, such that after deleting these vertices, the remaining graph is either isomorphic to $K_4$ or has degeneracy $\leq 2$. So by~\cref{lem:chi_c^star>4} and~\cref{lem:degeneracy}, these eight graphs have correspondence packing number at most $4$.
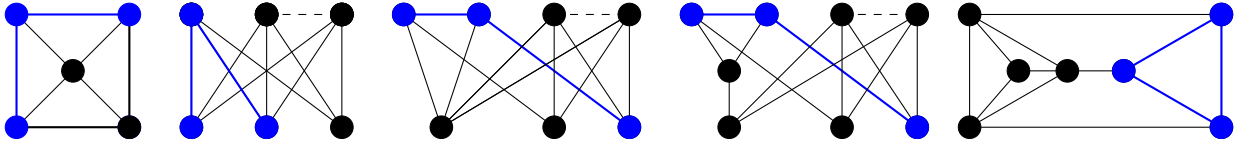
\begin{figure}[h]
 \centering
 \begin{tikzpicture}

 \draw[thick] (0,0)-- (1.5,0)--(1.5,1.5);

 \draw[blue, thick] (0,0)-- (0,1.5)--(1.5,1.5);

 \foreach \x in {0,1.5}{
 \foreach \y in {0,1.5}{
 
 \draw (\x,\y)--(0.75,0.75);
 
 }
 }

 \foreach \x in {0,1.5}{
 \foreach \y in {0,1.5}{
 
 \draw[fill, blue] (\x,\y) circle (0.15);
 
 }
 }
 \foreach \x/\y in {0.75/0.75,1.5/0}{
        \draw[fill] (\x,\y) circle (0.15);
 }

 \end{tikzpicture}
 \quad
 \begin{tikzpicture}

\draw[dashed] (2,1.5)--(3,1.5);

 \foreach \x in {1,2,3}{
 \foreach \y in {1,2,3}{
 
 \draw[fill] (\x,1.5) circle (0.15);
 \draw (\x,1.5)--(\y,0);
 }
 }

 \draw[blue, thick] (1,0)-- (1,1.5)--(2,0);

 \foreach \x in {1,2,3}{

 \draw[fill] (\x,1.5) circle (0.15);
 \draw[fill] (\x,0) circle (0.15);
 }

 \foreach \x/\y in {1/1.5,2/0,1/0}{

 \draw[fill, blue] (\x,\y) circle (0.15);
 }

 \end{tikzpicture}
 \quad
 \begin{tikzpicture}

 \draw (0,1.5)--(2,0);
\draw[blue, thick] (0,1.5)-- (1,1.5)--(3,0);

\draw[dashed] (2,1.5)--(3,1.5);

 \foreach \x in {0,1,2,3}{
 \draw[fill] (\x,1.5) circle (0.15);
 \draw (\x,1.5)--(0.5,0);
 }
 \foreach \x in {0.5,2,3}{
 \draw[fill] (\x,0) circle (0.15);
 \draw (2,1.5)--(\x,0)--(3,1.5);
 }

\foreach \x/\y in {1/1.5,3/0,0/1.5}{
    \draw[fill, blue] (\x,\y) circle (0.15);
 }

 \end{tikzpicture}
 \quad
 \begin{tikzpicture}

 \draw (0,1.5)--(2,0);
\draw[blue, thick] (0,1.5)-- (1,1.5)--(3,0);

\draw[dashed] (2,1.5)--(3,1.5);

 \foreach \x in {0,1,2,3}{
 \draw[fill] (\x,1.5) circle (0.15);
 }
 \foreach \x in {0.5,2,3}{
 \draw[fill] (\x,0) circle (0.15);
 \draw (2,1.5)--(\x,0)--(3,1.5);
 }
 \draw (0,1.5)--(0.5,0.75)--(1,1.5);
 \draw (0.5,0)--(0.5,0.75);
 \draw[fill] (0.5,0.75) circle (0.15);

 \foreach \x/\y in {1/1.5,3/0,0/1.5}{

 \draw[fill, blue] (\x,\y) circle (0.15);
 }
 
 \end{tikzpicture}\quad
 \begin{tikzpicture}[scale=0.75]

\draw (4.4641,0)--(0,0)--(0,2)--(1.73205,1)--(0,0);
\draw[blue, thick] (4.4641,0)-- (2.73205,1)--(4.4641,2)--(4.4641,0);
\draw (4.4641,2)--(0,2)--(0.866,1)--(0,0);
\draw (0.866,1)--(2.73205,1);

 \foreach \x in {0,4.4641}{
 \draw[fill] (\x,0) circle (0.2);
 \draw[fill] (\x,2) circle (0.2);
 }
 \foreach \x in {0.866,1.73205,2.73205}{
 \draw[fill] (\x,1) circle (0.2);
 }

 \foreach \x/\y in {4.4641/2,4.4641/0,2.73205/1}{

 \draw[fill, blue] (\x,\y) circle (0.2);
 }

 \end{tikzpicture}
 \caption{ $W_5, K_{3,3} \subset A, B \subset B^+, C \subset C^+$ and $D$ with blue induced $P_3$ or $C_3$ with degree-$3$ vertices}
 \label{fig:smallgraphs}
\end{figure}
\end{claimproof}

\begin{claim}
 $$\chi_c^\star(A^+) =\chi_c^{\star}(G_{6,2,3}^+)=\chi_c^{\star}(G_{7,3,3}^+)=4$$
\end{claim}

\begin{claimproof}
 Note that $A^+$ equals $K_{3,3}$ with two disjoint edges added, while the complement of $G_{6,2,3}^+$ is a $P_5$.
 Using a computer verification\footnote{See ~\cite[files \text{A+} and \text{G6+}]{CD23}.}, it has been verified that $\chi_c^\star(A^+)$ and $ \chi_c^{\star}(G_{6,2,3}^+)$ are bounded by $4.$\footnote{A manual verification seems hard, since not every partial packing of $2$ vertices of $K_4$ can be extended, nor can every partial packing of the induced $K_4$.}

The fact that $\chi_c^{\star}(G_{7,3,3}^+)=4,$ can be proven in a structured way. Let $G := G_{7,3,3}^+$ and $v_1, v_2, \ldots, v_7$ be the $7$ vertices of $G_{7,3,3}^+$ in order (the vertices of the straight path in~\cref{fig:3MainSmallexceptions}).

We will extend partial packings, where we first choose partial packings $\vec{c}(v_1), \vec{c}(v_7)$ for the two vertices of largest degree, next $\vec c(v_i)$ for $4 \le i \le 6,$ and finally $\vec c(v_2), \vec c(v_3).$

To be sure that we can extend the packing at the end, we need to choose $\vec{c}(v_7)$ already carefully.
Without loss of generality, we can assume that the (full) matchings on $v_1v_2$, $v_2v_3$ and $v_3v_7$ are the identity matchings and fix $\vec c(v_1)=(1,2,3,4).$
There are $9$ possible choices for $\vec c(v_7)$ (derangements of a permutation of $\vec{c}(v_1)$ determined by the matching on $v_1v_7$) to extend the partial packing to $G[v_1,v_7].$ In the graph $G[\{v_1,v_4,v_5,v_6,v_7\}],$ the path induced by $\{v_4,v_5,v_6\}$ only contains vertices of degree $3$ and thus any partial packing on $G[v_1,v_7]$ can be extended to those three vertices by~\cref{lem:chi_c^star>4}.

Depending on the matching on $v_1v _7$, there are $2,3,4$ or even $9$ (if the matching on $v_1v_7$ is the identity matching) choices for $\vec c (v_7)$ such that no matter what are the choices of $\vec c(v_i)$ for $4 \le i \le 6,$ the partial packing is extendable to the whole of $G$.\footnote{This statement is quickly checked by ranging over all possibilities, see~\cite[\text{Verifications\_NoInducedC3P3}]{CD23}.}
So there are choice of $\vec{c}(v_7)$ that the partial packing on $G[\{v_1,v_4,v_5,v_6,v_7\}]$ can be extended to the whole of $G$.
\end{claimproof}

\begin{figure}
 \centering
 \begin{tikzpicture}
 \draw (2,1.5)--(3,1.5);
 \draw (2,0)--(3,0);
 \foreach \x in {1,2,3}{
 \foreach \y in {1,2,3}{
 
 \draw[fill] (\x,1.5) circle (0.15);
 \draw (\x,1.5)--(\y,0);
 }
 }

 \foreach \x in {1,2,3}{

 \draw[fill] (\x,1.5) circle (0.15);
 \draw[fill] (\x,0) circle (0.15);
 }

 \end{tikzpicture}\quad
 \begin{tikzpicture}

 \foreach \x in {1,2,3,4,5,6}{

 \draw[fill] (\x,0) circle (0.15);
 }
\draw (6,0)--(1,0);
\foreach \x in {4,5,6}{
\Edge[style={bend right}](1,0)(\x,0)
}
\foreach \x in {4,3,2}{
\Edge[style={ bend right}](6,0)(\x,0)
}
 
 \end{tikzpicture}\quad
 \begin{tikzpicture}

 \foreach \x in {1,2,3,4,5,6,7}{

 \draw[fill] (\x,0) circle (0.15);
 }
\draw (7,0)--(1,0);
\foreach \x in {4,5,6,7}{
\Edge[style={ bend right}](1,0)(\x,0)
}
\foreach \x in {4,3,2}{
\Edge[style={bend right}](7,0)(\x,0)
}
 
 \end{tikzpicture}
 \caption{The graphs $A^+, G_{6,2,3}^+$ and $G_{7,3,3}^+$}
 \label{fig:3MainSmallexceptions}
\end{figure}
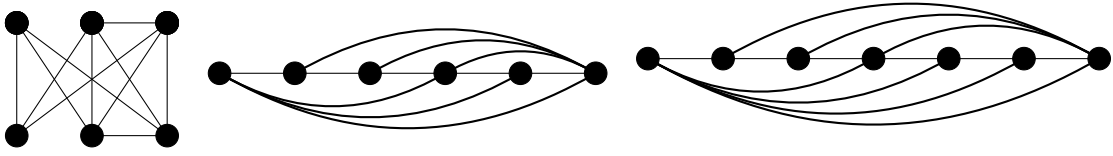
\end{proof}

\begin{lem}\label{lem:K_5^-}
 $\chi_c^{\star}(K_5^{-})= 5$
\end{lem}

\begin{proof}
 Let $(L,H)$ be a $5$-correspondence-cover of 
 $K_5^-$. Let $u$ and $v$ be the two vertices with degree $3$, and $w$ be a vertex with degree $4.$ We can first choose a packing $\vec c$ of the remaining two vertices.
 By~\cref{lem:Nofderangements} and \cref{cor:NoCommonDerangement}, there are at least $12$ options to extend the packing to $w$, while for both $u$ and $v$, there are at most $2$ choices for $\vec c(w)$ which make extending the partial packing to that vertex impossible.
 As such, there are at least $8$ choices for $\vec c(w)$ such that subsequently $\vec c(u)$ and $\vec c(v)$ can be chosen to extend the packing, implying that there exists a proper $L$-packing $\vec c$ of $G$.
 So $\chi_c^{\star}(K_5^{-}) \le 5$.

 By checking all $4$-fold correspondence-covers of $K_5^-$, we conclude that there is a unique (up to isomorphism) $4$-fold correspondence-cover which does not have a packing, implying that $\chi_c^\star(K_5^-)>4.$ See~\cref{fig:K5-cover_corr5} for an illustration of that cover. The dashed lines are full identity matchings, while the full lines in $K_5^-$ correspond with a matching determined by a permutation $(123)$.
 As verified in~\cite[\text{K5-}]{CD23}, the cover has exactly $54$ proper correspondence-colorings (maximum independent sets in the cover graph $H$) and no four of them pack. Hence $\chi_c^\star(K_5^-)=5.$
\begin{figure}[h]
 \centering
 \begin{tikzpicture}

 \foreach \x in {0,144,216,288}{
 
 \draw[gray, dashed] (\x+90:2)--(\x-54:2);
 
 }

 \draw[gray, dashed] (216+18:2)--(-54:2);

 \foreach \x in {0,72,144,216,288}{

 \draw[fill] (\x+18:2) circle (0.15);
 }

 \foreach \x in {0,72,144,216}{
 
 \draw[thick] (\x+18:2)--(\x-54:2);
 
 }

 \end{tikzpicture}\quad
 \begin{tikzpicture}[scale=0.75]
 
\foreach \x in {0,72,144,216,288}{
\draw[fill=gray!17, rotate around={\x+18:(\x+18:2.75)}] (\x+18:2.75) ellipse (40pt and 15pt);

}

\foreach \x in {0,144,216,288}{
 \foreach \r in {2,2.5,3,3.5}{
 \draw[gray!50, dashed] (\x+90:\r)--(\x-54:\r);
 }
 }
 \foreach \r in {2,2.5,3,3.5}{
 \draw[gray, dashed] (216+18:\r)--(-54:\r);
 }

 \foreach \x in {0,72,144,216,288}{
 \foreach \r in {2,2.5,3,3.5}{

 \draw[fill] (\x+18:\r) circle (0.15);
 
 }
 }

\foreach \x in {0,72}{
 \draw (\x+18:2.5)--(\x-54:3);
 \draw (\x+18:3)--(\x-54:3.5);
 \draw (\x+18:3.5)--(\x-54:2.5);
 \draw (\x+18:2)--(\x-54:2);
 }

 \foreach \x in {144,-144}{
 \draw (\x+18:3.5)--(\x-54:3);
 \draw (\x+18:3)--(\x-54:2.5);
 \draw (\x+18:2.5)--(\x-54:3.5);
 \draw (\x+18:2)--(\x-54:2);
 }

 \end{tikzpicture}
 \caption{$K_5^-$ with its unique $4$-fold cover graph $H$ of chromatic number $>4$. Dashed and bold lines differentiate the two types of matchings that occur. 
 %The cliques on the lists are omitted in this drawing.
 }
 \label{fig:K5-cover_corr5}
\end{figure}
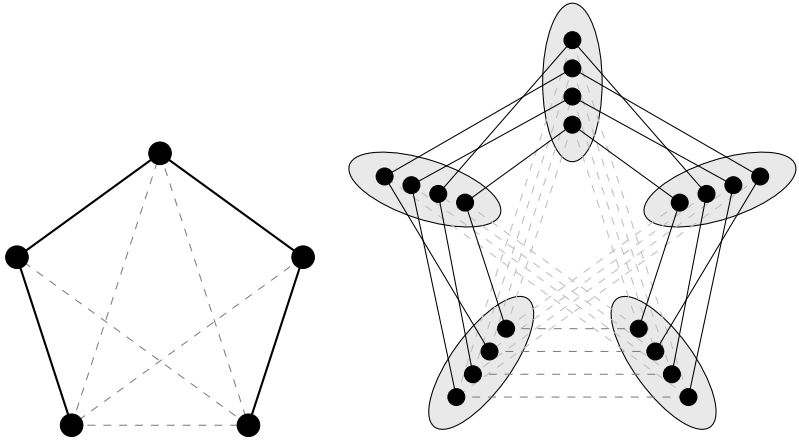
 \end{proof}

\begin{thr}\label{thr:K24minorfree+}
 Let $G$ be a $K_{2,4}$-minor-free graph with at least one edge. Then
 $$ \chi_{c}^{\star}(G) = \begin{cases}
 6 & \text{ if } G \text{ contains a } K_5, \\
 5 & \text{ if } G \text{ contains a } K_5^- \text{ and no } K_5, \\
 4 & \text{ if $G$ contains a cycle and no $K_5^-$, } \\
 2 & \text{ if $G$ is a forest.}
\end{cases}$$
\end{thr}

\begin{proof}
 As before, we may assume that $G$ is $2$-connected.
 If $G$ has $K_5$ as a subgraph then $G=K_5$, as adding any path connecting two vertices of $K_5$ results in a graph that contains $K_{2,4}$ as a minor.
 As a corollary of~\cite[Thr.~4]{CCDK23} and a classical construction (compare~\cite{Yus21}), $\chi_c^\star(K_5)=6$.

 Assume $G$ is not $K_5$ but contains $K_5^{-}$ as a subgraph. Then $\chi_c^{\star}(G) \ge \chi_c^{\star}(K_5^-) = 5$ by Lemma~\ref{lem:K_5^-}. We shall show that $\chi_c^{\star}(G) = 5$. Assume to the contrary that $\chi_c^{\star}(G) > 5$. By choosing a minimum such graph $G$, we may assume that $\chi_c^{\star}(G\setminus v) \le 5$ for any vertex $v$ of $G$. 
 Assume the copy of $K_5^-$ has vertices $v_1,v_2,\ldots, v_5$ and $v_1v_2 \notin E(G)$. 
 Then $v_1,v_2$ are in distinct connected components of $G\setminus \{v_3,v_4,v_5\}$, for otherwise $G$ contains $K_{2,4}$ as a minor. So $G$ is obtained from $K_5^-$ by replacing an edge $uv$ with an $uv$-outerplanar graph. But then $G$ contains a degree 2 vertex $v$. Since $\chi_c^{\star}(G\setminus v) \le 5$, we have $\chi_c^{\star}(G) \le 5$ by~\cref{lem:degeneracy}, a contradiction.
 
 Assume $G$ does not contain $K_5^-$ as a subgraph. 
 If $G$ is a forest, then $\chi_c^{\star}(G) =2$. Otherwise $G$ contains a cycle $C$ and hence $\chi_c^{\star}(G) \ge \chi_c^{\star}(C)=4$. It follows from ~\cref{lem:K24minorfree+} that $\chi_c^{\star}(G) = 4$.
\end{proof}

\begin{rem}
 Since $\chi_{\ell}^\bullet(W_n)=4$ for every even $n$ and $\chi_c^\bullet(K_5^-)=\chi_\ell^\star(K_5^-)=4$, for each parameter $\rho \in \{ \chi_c^\bullet, \chi_\ell^\bullet, \chi_\ell ^\star\}$, a small adaptation of~\cref{lem:K24minorfree+} yields $\rho(G)\le 4$ except if $G=K_5.$
\end{rem}

\begin{rem}
 In the proof of~\cref{thr:K24minorfree+}, it is shown that $\chi_c^\star(K_5^-)=5.$
 Related to the inverse problem for $\chi_c^\star$, see~\cite{CH23}, one can wonder about the value of $\chi_c^\star(K_n^-)$ for larger odd $n$.
\end{rem}

\section{Concluding remarks and open problems}\label{sec:conc}

Our results and methods indicate that (fractional) packing numbers provide a powerful framework, with immediate implications for $\epsilon-$flexibility and list-colorability. 
While for fractional packing we were able to employ global arguments, for integral packing most proofs depend on locally extending partial packings. Some of them use simple counting methods, where we count the number of valid extensions in the vertex itself and compare with the number of bad choices for future extensions in other graphs.
This can be considered as a local version of a current trend of obtaining short colorability proofs via a stronger induction hypothesis involving counting the proper colorings (e.g.~\cite{PH21,BBCK23}). It is not inconceivable that there are, at this point unknown, connections between the number of colorings and packing of colorings that can be leveraged for even more efficient induction hypotheses.

A glance at the gaps in Table~\ref{table:overview} reveals that many questions remain.
In particular, in line of~\cite[Conj.~23]{CCDK23}, it is still in the realm of possibilities that Theorems~\ref{thm:list_planar} and~\ref{thm:exponential} fully generalize to the \textit{fractional} packing numbers, i.e. that $\chi_\ell^{\bullet}(G)$ is at most resp. $5,4,3$ for the classes of graphs containing those that are planar, triangle-free planar, planar of girth at least 5. 
Encouraged by the fact~\cite{DMMP21} that the weighted flexibility for the class of planar triangle-free graphs is at most $4$, we conjecture:
\begin{conj}
 Let $G$ be a planar triangle-free graph. Then $\chi_\ell^\bullet(G) \le 4$. 
\end{conj}

It is known~\cite{DNP19} that for every $d\geq 0$, there exists $\epsilon:=(d+2)^{-(d+1)^2}>0$ such that $d$-degenerate graphs are weighted $\epsilon$-flexible for lists of size $d+2$. As planar graphs are $5$-degenerate, this gives the state-of-the art for planar graphs mentioned in Theorem~\ref{thm:listpacking}. The value of $\epsilon$ is small however. This would immediately be improved if $\chi_\ell^{\bullet}(G)\leq d+2$ for every $d$-generate graph $G$. 
A more modest goal is to prove $\chi_\ell^\bullet(G) \le 7$ for every planar graph $G$.

With regards to the (integral) packing numbers we are inclined to exercise more caution. Indeed, as our construction for girth $5$ planar graphs in Theorem~\ref{thm:listpacking} demonstrated, these numbers are potentially larger.
We ask:

\begin{itemize}
 \item What is the maximum value of $\chi_\ell^\star(G)$ among all bipartite planar graphs $G$?
 \item What is the maximum value of $\chi_\ell^\star(G)$ among all planar triangle-free graphs $G$?
 \item What is the maximum value of $\chi_\ell^\star(G)$ among all planar graphs $G$?
\end{itemize}

As a strengthening of~\cite[Cor.~3.4]{AT92}, one could wonder if $\chi_\ell^\bullet(G) \le 3$ for every bipartite planar $G$.
Naturally, all of the above questions and conjectures also have a correspondence analogue.\\

It was conjectured in~\cite{CCDK23} that $\chi_c^{\star}(G) \leq 2 \left\lceil \frac{\Delta(G)+1}{2} \right\rceil$, where $\Delta(G)$ is the maximum degree of $G$. Inspired by our~\cref{thm:MAD<4ImpliesCorrPack<=5} and~\cref{thm:MAD<6ImpliesCorrPack<=8} and our suspicion that $\chi_c^{\star}(G)\leq 7$ for planar graphs, we ask something stronger: is it true that for any $k \in \mathbb N,$
$\mad(G)<k$ implies $\chi_c^{\star}(G)\leq k+1$? 
This should be contrasted with the fact that there exist $k$-degenerate graphs $G$ with $\chi_c^{\star}(G)=2k$.

We end with a simple observation which may further motivate the study of the fractional list packing number.
Albertson, Grossman and Haas~\cite{AGH00} conjectured that if $k'<k:=\chi_{\ell}(G)$, then for every $k'$-fold list-assignment $L'$ of $G$, there is a partial $L'$-coloring that colors at least $k' n/ k$ vertices of $G$. We note that this becomes true if one takes $k:=\chi_{\ell}^{\bullet}(G)$ instead. This is because if one takes a random $k$-fold $L$-coloring $c$ according to the probability distribution guaranteed by the fractional list-packing, one obtains $\mathbb{E} \#\{v\in V(G) \mid c(v) \in L^{'}(v)\}= \sum_{v\in V(G)} \mathbb{P}(c(v)\in L^{'}(v)) = k'n/k$. This implies that the conjecture of Albertson, Grossman and Haas is true for every graph $G$ with $\chi_{\ell}^{\bullet}(G)=\chi_{\ell}(G)$.

\paragraph{Note.} Dedicated to the memory of S.C.'s late housemate Tim Eppink.

\paragraph{Open access.} For the purpose of open access, a CC BY public copyright licence is applied to any Author Accepted Manuscript (AAM) arizing from this submission.

\bibliographystyle{habbrv}
\bibliography{listpack}

\end{document}